\documentclass[11pt]{article}
\title{3-Preprojective Algebras of Type D}
\author{Jordan Haden\\University of East Anglia\\jordan.haden@uea.ac.uk}

\usepackage[a4paper, margin=1in]{geometry}
\usepackage{mathtools, amssymb, amsthm, setspace, cancel, pifont, tikz-cd}

\theoremstyle{definition}
\newtheorem{defn}{Definition}[section]
\newtheorem{example}[defn]{Example}

\theoremstyle{remark}
\newtheorem{remark}[defn]{Remark}
\newtheorem{notation}[defn]{Notation}

\theoremstyle{plain}
\newtheorem{theorem}[defn]{Theorem}
\newtheorem{prop}[defn]{Proposition}
\newtheorem{lemma}[defn]{Lemma}
\newtheorem{corollary}[defn]{Corollary}
\newtheorem{theorem*}{Theorem}

\newcommand{\nn}{\mathbb{N}}
\newcommand{\zz}{\mathbb{Z}}
\newcommand{\cc}{\mathbb{C}}
\newcommand{\rr}{\mathbb{R}}

\newcommand{\dd}{\mathcal{D}}
\newcommand{\ee}{\mathcal{E}}
\newcommand{\mm}{\mathcal{M}}
\newcommand{\uu}{\mathcal{U}}
\renewcommand{\aa}{\mathcal{A}}
\renewcommand{\tt}{\mathcal{T}}

\newcommand{\jac}{\operatorname{\mathcal{J}}}
\newcommand{\id}{\operatorname{id}}
\newcommand{\aut}{\operatorname{Aut}}
\newcommand{\im}{\operatorname{Im}}
\newcommand{\soc}{\operatorname{Soc}}

\newcommand{\tr}{\operatorname{tr}}
\newcommand{\gldim}{\operatorname{gl.dim}}
\newcommand{\add}{\operatorname{add}}
\newcommand{\ext}{\operatorname{Ext}}
\newcommand{\derb}{\operatorname{\mathbf{D}^b}}
\newcommand{\parity}{\operatorname{par}}
\newcommand{\cof}{\operatorname{coef}}
\newcommand{\rad}{\operatorname{rad}}

\renewcommand{\mod}{\operatorname{mod}}
\renewcommand{\hom}{\operatorname{Hom}}

\newcommand{\cmark}{\ding{51}}
\newcommand{\xmark}{\ding{55}}

\begin{document}

\maketitle

\begin{abstract}
We present a family of selfinjective algebras of type D, which arise from the 3-preprojective algebras of type A by taking a $\mathbb{Z}_3$-quotient. We show that a subset of these are themselves 3-preprojective algebras, and that the associated 2-representation-finite algebras are fractional Calabi-Yau. In addition, we show our work is connected to modular invariants for SU(3).
\end{abstract}



\section{Introduction}

\subsection*{Overview}

Many mathematical objects admit a classification in terms of Dynkin diagrams, perhaps the most famous example being complex semisimple Lie algebras. A theorem of Gabriel says that the path algebra of a quiver has finite representation type if and only if its underlying graph is an ADE Dynkin diagram. The group $\zz_2$ acts on the type A diagrams by rotating them through $\pi$. One obtains the type D diagrams through quotienting by this action, duplicating the fixed vertex whenever one exists (see Table \ref{Table: Quotients of Dynkin diagrams}).

\begin{table}
\centering
    \begin{tabular}{c|c}
        Type A Dynkin diagram & $\zz_2$-quotient\\
        \hline
        $A_2:\ \begin{tikzcd} \bullet \arrow[r, no head] & \bullet \end{tikzcd}$ & $T_1:\ \begin{tikzcd} \bullet \arrow[no head, loop, distance=2em, in=35, out=325] \end{tikzcd}$\\
        $A_3:\ \begin{tikzcd} \bullet \arrow[r, no head] & \circ \arrow[r, no head] & \bullet \end{tikzcd}$ & $D_3:\ \begin{tikzcd}
                            & \circ                     \\
        \bullet \arrow[ru, no head] &                             \\
                            & \circ \arrow[lu, no head]
        \end{tikzcd}$\\
        $A_4:\ \begin{tikzcd} \bullet \arrow[r, no head] & \bullet & \bullet \arrow[l, no head] & \bullet \arrow[l, no head] \end{tikzcd}$ & $T_2:\ \begin{tikzcd} \bullet \arrow[r, no head] & \bullet \arrow[no head, loop, distance=2em, in=35, out=325] \end{tikzcd}$\\
        $A_5:\ \begin{tikzcd} \bullet \arrow[r, no head] & \bullet \arrow[r, no head] & \circ \arrow[r, no head] & \bullet \arrow[r, no head] & \bullet \end{tikzcd}$ & $D_4:\begin{tikzcd}
                           &                                                 & \circ \\
            \bullet \arrow[r, no head] & \bullet \arrow[ru, no head] \arrow[rd, no head] &       \\
                           &                                                 & \circ
    \end{tikzcd}$\\
        \vdots & \vdots
    \end{tabular}
    \caption{Type D Dynkin diagrams (and tadpole diagrams) arising as $\zz_2$-quotients of type A diagrams.}
    \label{Table: Quotients of Dynkin diagrams}
\end{table}

A related classification appears in \cite{evans2009ocneanu, evans2012nakayama}, where Evans and Pugh study Jacobian algebras of so-called $\aa\dd\ee$ graphs, introduced by Di Francesco and Zuber in work on SU(3) modular invariants \cite{di1990n}. The group $\zz_3$ acts on the type $\aa$ graphs by rotating them through $2\pi/3$. The type $\dd$ graphs arise from type $\aa$ by taking $\zz_3$-orbifolds, which amounts to quotienting by the action and triplicating the fixed vertex, whenever one exists. The type $\aa$ algebras are well-studied. Indeed, in \S\ref{Section: Operator algebras} we show they are isomorphic to the 3-preprojective algebras of type A. This article can be seen as an exploration of  the type $\dd$ algebras from the perspective of higher homological algebra.

In \cite{iyama2011cluster}, Iyama introduced $d$-representation-finite algebras: algebras which have global dimension at most $d$, and whose module category has a $d$-cluster tilting subcategory. One can define higher analogues of the Auslander-Reiten translates which restrict to this subcategory, and hence gain some understanding of the representation theory of the algebra, even if it has wild representation type. 

Herschend and Iyama showed that $d$-representation-finiteness is closely linked to the fractional Calabi-Yau property \cite{herschend2011n}. If an algebra $\Lambda$ has finite global dimension, the bounded derived category of its module category has a Serre functor, an autoequivalence satisfying a certain duality. If a power of the Serre functor is given by a shift, $\Lambda$ is said to be fractional Calabi-Yau. This property was introduced by Kontsevich to generalise properties of Calabi-Yau manifolds, which are important in theoretical physics.  

Given a $d$-representation-finite algebra $\Lambda$, one can construct its $(d+1)$-preprojective algebra $\Pi$. Note that $\Pi$ is always selfinjective, and if $\Lambda$ is basic then $\Pi$ is Frobenius, meaning it is isomorphic to its dual as a $\Pi$-$\Pi$-bimodule, provided one twists by some automorphism $\sigma$. We call $\sigma$ the Nakayama automorphism of $\Pi$. Grant showed that a $d$-representation-finite algebra is fractional Calabi-Yau if and only if the Nakayama automorphism of its $(d+1)$-preprojective algebra has finite order \cite{grant2022serre}.

In this article we present a family of selfinjective algebras we call type D, which are Morita equivalent to skew group algebras of the 3-preprojective algebras of type A. Our definition is standalone, in the sense that it makes no reference to type A, and we prove the Morita equivalence by showing that our definition agrees with a construction of Giovannini and Pasquali \cite{giovannini2019skew}. In \S\ref{Section: Operator algebras} we construct isomorphisms between our algebras and the type $\dd$ algebras of Evans and Pugh.

We show that one in three of the selfinjective algebras of type D are 3-preprojective. By considering their Nakayama automorphisms, we show that the corresponding 2-representation-finite algebras are fractional Calabi-Yau. Finally, in \S\ref{Section: 2AR Quivers} we give recipes to construct 2-Auslander-Reiten quivers for these algebras, on which one can see the fractional Calabi-Yau property quite explicitly.

\begin{notation}
Following the convention in \cite{evans2009ocneanu, evans2012nakayama}, we denote by $\dd^s$ the quiver which arises from $\aa^s$ (and label the corresponding algebras accordingly). However, only every third quiver is the quiver of a 3-preprojective algebra. We could have chosen to only label those with a $\dd$. However, it is not obvious how one should index in this case (see Table \ref{Table: Possibilities for labelling higher type D quivers}).
\end{notation}

\begin{table}
    \centering
    \begin{tabular}{c|c|c|c}
        Quiver & Our notation & 3-preprojective algebra? & Alternative notation\\
        \hline
        $\begin{tikzcd} \bullet \arrow[loop, distance=2em, in=325, out=35] \end{tikzcd}$ & $\dd^2$ & \xmark & $\tt^?$\\
        $\begin{tikzcd} \bullet \arrow[r, shift left] & \bullet \arrow[l, shift left] \arrow[loop, distance=2em, in=325, out=35] \end{tikzcd}$ & $\dd^3$ & \xmark & $\tt^?$\\
        - & $\dd^4$ & \cmark & $\dd^?$\\
        - & $\dd^5$ & \xmark & $\tt^?$\\
        - & $\dd^6$ & \xmark & $\tt^?$\\
        - & $\dd^7$ & \cmark & $\dd^?$\\
        \vdots & \vdots & \vdots & \vdots
    \end{tabular}
    \caption{Two options for notation. See Figure \ref{Figure: Examples of D^s} for the missing quivers.}
    \label{Table: Possibilities for labelling higher type D quivers}
\end{table}

\subsection*{Main results}

By \emph{algebra} we mean  associative, unital, finite-dimensional algebra over $k=\cc$. Modules are taken to be right modules unless otherwise stated. If $p,q$ are paths in some quiver, $pq$ means ``first $p$ then $q$". We take $\nn$ to contain 0.

Denote by $\succ$ the strict lexicographic order on $\nn^3$. Let $\omega\colon\nn^3\to\nn^3$, $\omega(x_0,x_1,x_2)=(x_1,x_2,x_0)$.

\begin{defn}
\label{Defn: D^s}
Let $s\in\zz$, $s\ge2$. Define 
\begin{align*}
    Q_0&=\{x\in\nn^3\mid x_0+x_1+x_2=s-1,\ x\succ\omega(x),\ x\succ\omega^2(x)\},\\
    Q_1&=\bigcup_{i,j=0}^2\{\alpha_{i,j}\colon x\longrightarrow\omega^j(x)+f_i\mid x,\omega^j(x)+f_i\in Q_0\},
\end{align*}
 where $f_0=(-1,1,0)$, $f_1=(0,-1,1)$ and $f_2=(1,0,-1)$.

If $s\not\equiv1$ (mod 3), let $\dd^s$ be the quiver with vertices $\dd_0^s=Q_0$ and arrows $\dd^s_1=Q_1$.

If $s=3t+1$ for some $t\in\zz^+$, write $X=(t,t,t)$, and let $\dd^s$ be the quiver with vertices \[\dd^s_0=Q_0\cup\{X_0, X_1, X_2\}\] (i.e. take three copies of $X$ indexed by $\{0,1,2\}$) and arrows \[\dd^s_1=Q_1\cup\{\beta_k\colon X-f_0\longrightarrow X_k,\ \gamma_k\colon X_k\longrightarrow X+f_2\mid k=0,1,2\}.\]
\end{defn}

\begin{notation}
An arrow $\alpha_{i,j}\colon x\longrightarrow\omega^j(x)+f_i$ in $\dd^s$ is uniquely determined by its source and indices, so we often denote it $e_x\alpha_{i,j}$. If $j=0$ we simplify notation further and just write $e_x\alpha_i$. Our convention when drawing $\dd^s$ is to label the ``$\alpha$" arrows using their indices alone.
\end{notation}

 Some examples of $\dd^s$ are presented in Figure \ref{Figure: Examples of D^s}.

\begin{figure}
\centering
\includegraphics{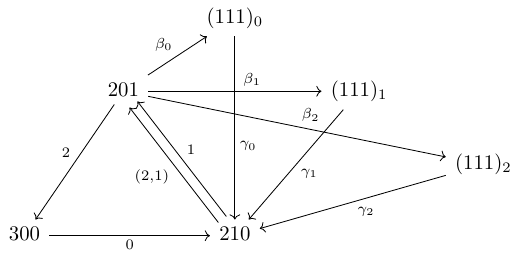}\\
\includegraphics{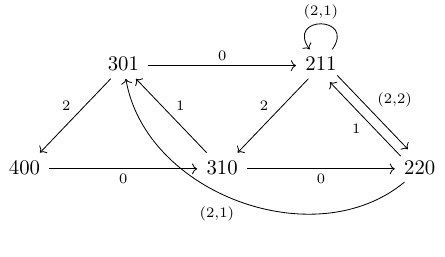}\\
\includegraphics{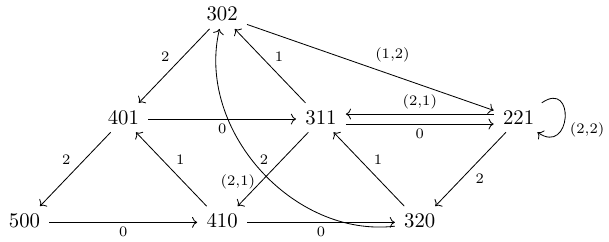}\\
\includegraphics[scale=0.83]{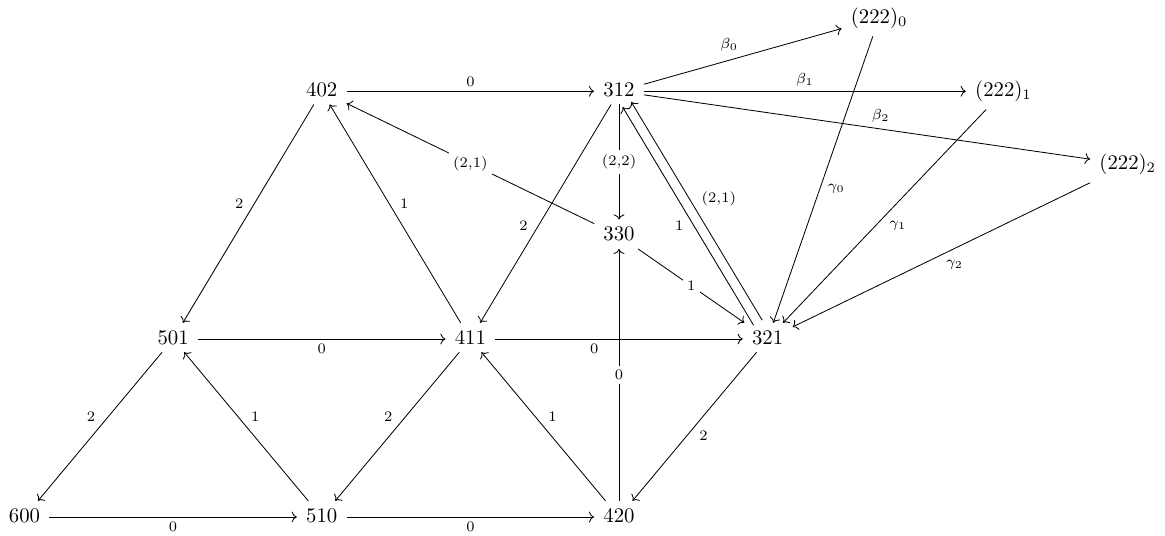}
\caption{Top to bottom: $\dd^4,\dd^5,\dd^6,\dd^7$.}
\label{Figure: Examples of D^s}
\end{figure}

Let $(Q,W)$ be a quiver with potential (QP) - a quiver together with a linear combination of cycles. By formally differentiating $W$ with respect to the arrows of $Q$, one obtains the \emph{Jacobian algebra} $\jac(Q,W)=kQ/\langle\partial_\alpha W\mid\alpha\in Q_1\rangle$. See e.g. \cite[\S2.1]{bocklandt2008graded} for a full exposition.

\begin{defn}
\label{Defn: Potential on D^s}
For each $s\ge2$, we define a Jacobian algebra $\Pi_\dd^s=\jac(\dd^s,W_\dd^s)$ over $\cc$. The potential is \[W_\dd^s=\sum_c\lambda_\dd(c)c,\] where the sum is taken over all 3-cycles in $\dd^s$, and $\lambda_\dd$ is defined as follows.

For $x,y\in\nn^3$, write $x\sim y$ if $x=y$, $\omega(x)=y$ or $\omega^2(x)=y$. Take expressions involving indices of arrows mod 3.
\begin{enumerate}
    \item Suppose $c=e_x\alpha_{i_0,j_0}\alpha_{i_1,j_1}\alpha_{i_2,j_2}$ is a product of three distinct arrows. If $j_0+j_1+j_2=0$ then \[\lambda_\dd(c)=
    \begin{cases}
        1&\text{if }(i_0+j_0,i_1-j_2,i_2)\sim(0,1,2),\\
        -1&\text{if }(i_0+j_0,i_1-j_2,i_2)\sim(0,2,1).
    \end{cases}\]
    Otherwise, $\lambda_\dd(c)=0$.
    \item If $c=e_x\alpha_{i,j}^3$ for a loop $e_x\alpha_{i,j}$ then \[\lambda_\dd(c)=
    \begin{cases}
        \frac{1}{3}&\text{if }(i+j,i-j,i)\sim(0,1,2),\\
        -\frac{1}{3}&\text{if }(i+j,i-j,i)\sim(0,2,1).
    \end{cases}\]
    \item If $s\equiv1$ (mod 3) then for each $k\in\{0,1,2\}$, \[\lambda_\dd(\alpha_1\beta_k\gamma_k)=-1,\ \lambda_\dd(\alpha_{2,1}\beta_k\gamma_k)=\zeta^k,\] where $\zeta=e^{2\pi i/3}$.
\end{enumerate}
\end{defn}

\begin{example}
\label{Example: Potentials on D^4, D^5}
The potential on $\dd^4$ is \[W_\dd^4=\alpha_0\alpha_1\alpha_2+\cancel{0\alpha_0\alpha_{2,1}\alpha_2}+\sum_{k=0}^2\left(\zeta^k\alpha_{2,1}\beta_k\gamma_k-\alpha_1\beta_k\gamma_k\right).\] Hence $\Pi_\dd^4$ is the algebra $\cc\dd^4$ modulo the relations
\begin{align*}
    \alpha_1\alpha_2, \;\; & \;\; \alpha_0\alpha_1,\\
    \sum_{k=0}^2\zeta^k\beta_k\gamma_k, \;\; & \;\; \alpha_2\alpha_0-\sum_{k=0}^2\beta_k\gamma_k,\\
    \{\zeta^k\gamma_k\alpha_{2,1}-\gamma_k\alpha_1\mid k=0,1,2\}, \;\; & \;\; \{\zeta^k\alpha_{2,1}\beta_k-\alpha_1\beta_k\mid k=0,1,2\}.
\end{align*}
The potential on $\dd^5$ is given by \[W_\dd^5=e_{400}\alpha_0\alpha_1\alpha_2+e_{310}\alpha_0\alpha_1\alpha_2-e_{301}\alpha_0\alpha_2\alpha_1+e_{301}\alpha_0\alpha_{2,2}\alpha_{2,1}-e_{211}\alpha_{2,1}\alpha_{2,2}\alpha_1+\frac{1}{3}e_{211}\alpha_{2,1}^3.\]
\end{example}

In \S\ref{Section: Type A} we recall the 3-preprojective algebras of type A, which we denote $\Pi_\aa^s$. The group $\zz_3$ acts on these by automorphisms, so we can consider the skew group algebras $\Pi_\aa^s\#\zz_3$. In \S\ref{Section: Morita equivalence} we apply a construction of \cite{giovannini2019skew} to prove the following.

\begin{theorem*}[\ref{Theorem 1 body}]
\label{Theorem 1}
For each $s\ge2$, $\Pi^s_\dd$ is Morita equivalent to $\Pi^s_\aa\#\zz_3$.
\end{theorem*}

In type A, one can always take cuts of $\Pi_\aa^s$ to obtain 2-representation-finite algebras $\Pi_\aa^s/\langle C\rangle$, which are fractional Calabi-Yau. In \S\ref{Section: Cuts} we show that the situation is different in type D: it is possible to take cuts of $\Pi_\dd^s$ to obtain 2-representation-finite algebras if and only if $s\equiv1$ (mod 3). 

\begin{theorem*}[\ref{Theorem 2 body}]
\label{Theorem 2}
If $s\equiv1$ (mod 3) and $C\subset\dd^s_1$ is a cut, then $\Pi_\dd^s/\langle C\rangle$ is 2-representation-finite, and its 3-preprojective algebra is $\Pi_\dd^s$.
\end{theorem*}

In \S\ref{Section: Calabi-Yau} we prove that the resulting 2-representation-finite algebras are fractional Calabi-Yau. 

\begin{theorem*}[\ref{Theorem 3 body}]
\label{Theorem 3}
Let $s=3t+1$, where $t\in\zz^+$. For any cut $C\subset\dd^s_1$, $\Pi_\dd^s/\langle C\rangle$ is fractional Calabi-Yau of dimension $2t/(t+1)$.
\end{theorem*}

\section{3-preprojective algebras of type A}

\label{Section: Type A}

We recall some background on $d$-representation-finite algebras and their $(d+1)$-preprojective algebras, before presenting an important example.

\begin{defn}
\label{Defn: d-representation-finite}
\cite[Def 2.1-2]{iyama2011n} Let $\Lambda$ be an algebra, $d\in\zz^+$. We call $M\in\mod\Lambda$ a \textit{$d$-cluster tilting object} if
\begin{align*}
    \add M&=\{N\in\mod\Lambda\mid \ext^i_\Lambda(M,N)=0\ \forall i\in\{1,\dots,d-1\}\}\text{ and}\\
    \add M&=\{N\in\mod\Lambda\mid\ext^i_\Lambda(N,M)=0\ \forall i\in\{1,\dots,d-1\}\}.
\end{align*}
We call $\Lambda$ \textit{$d$-representation-finite} if $\gldim\Lambda\le d$ and $\mod\Lambda$ contains a $d$-cluster tilting object.
\end{defn}

Let $\Lambda$ be a $d$-representation-finite algebra. The \emph{$d$-Auslander-Reiten translates} are \cite[Thm 1.4.1]{iyama2007higher} 
\begin{align*}
\tau_d&\coloneqq\tau\Omega^{d-1}\colon\underline{\mod}\Lambda\to\overline{\mod}\Lambda,\\
\tau_d^-&\coloneqq\tau^-\Omega^{1-d}\colon\overline{\mod}\Lambda\to\underline{\mod}\Lambda,
\end{align*}
where $\tau$, $\tau^-$ are the classical Auslander-Reiten translates and $\Omega$, $\Omega^-$ are the syzygy and cosyzygy functors, respectively (see e.g. \cite[\S IV.2]{assem2006elements}). These translates allow us to define a generalisation of the classical preprojective algebra of a quiver.

\begin{defn}
\cite[\S2]{herschend2011n} Let $\Lambda$ be a $d$-representation-finite algebra. The \emph{$(d+1)$-preprojective algebra} of $\Lambda$ is \[\Pi(\Lambda)=\bigoplus_{i\ge0}\hom_\Lambda(\Lambda,\tau_d^{-i}\Lambda).\] If $f\colon\Lambda\to\tau_d^{-i}\Lambda$ and $g\colon\Lambda\to\tau_d^{-j}\Lambda$, their product is \[gf=\tau_d^{-i}(g)\circ f\colon\Lambda\to\tau_d^{-(i+j)}\Lambda.\] There is a natural $\zz$-grading on $\Pi(\Lambda)$, called the \emph{tensor grading}, where the degree $i$ part is $\hom_\Lambda(\Lambda,\tau_d^{-i}\Lambda)$.
\end{defn}

2-representation-finite algebras are particularly well-understood, thanks to a result of Herschend and Iyama. We need the following notion.

\begin{defn}
\cite[Def 3.1]{herschend2011selfinjective}
If $(Q,W)$ is a QP, then to each subset $C\subseteq Q_1$ we associate a grading $g_C$ on $kQ$, given on arrows by \[g_C(\alpha)=
\begin{cases}
    1 & \text{if } \alpha\in C,\\
    0 & \text{else.}
\end{cases}\]
A subset $C\subseteq Q_1$ is called a \textit{cut} if $W$ is homogeneous of degree 1 with respect to $g_C$. If $C$ is a cut then $g_C$ induces a grading on $\jac(Q,W)$, and we call the degree 0 part $\jac(Q,W)_C$ a \textit{truncated Jacobian algebra}.
\end{defn}

Note that $\jac(Q,W)_C\cong\jac(Q,W)/\langle C\rangle$.

As in \cite[Def 3.6]{herschend2011selfinjective}, we call a QP $(Q,W)$ \textit{selfinjective} if $\jac(Q,W)$ is a selfinjective algebra. For general background on selfinjective, Frobenius and symmetric algebras see e.g. \cite{farnsteiner2005self}.

\begin{theorem}
\label{Theorem: Classification of 2-rep-finite algebras}
\cite[Thm 3.11]{herschend2011selfinjective} 
\begin{enumerate}
    \item If $(Q,W)$ is a selfinjective QP and $C\subseteq Q_1$ is a cut, then $\jac(Q,W)_C$ is 2-representation-finite, and its 3-preprojective algebra is $\jac(Q,W)$.
    \item Every basic 2-representation-finite algebra arises this way.
\end{enumerate}
\end{theorem}

If $\jac(Q,W)_C$ is basic 2-representation-finite, then the grading $g_C$ on $\jac(Q,W)$ corresponds to the tensor grading.

We now recall the 3-preprojective algebras of type A.

\begin{defn}
\cite[Def 5.1]{iyama2011n}
Let $s\in\zz$, $s\ge2$. Let $\aa^s$ be the quiver with vertices \[\aa^s_0=\{x\in\nn^3\mid x_0+x_1+x_2=s-1\}\] and arrows \[\aa_1^s=\bigcup_{i=0}^2\{\alpha_i\colon x\longrightarrow x+f_i\mid x,x+f_i\in\aa^s_0\},\] where $f_0=(-1,1,0)$, $f_1=(0,-1,1)$ and $f_2=(1,0,-1)$.
\end{defn}

For example, $\aa^4$ is the following quiver. \[\includegraphics{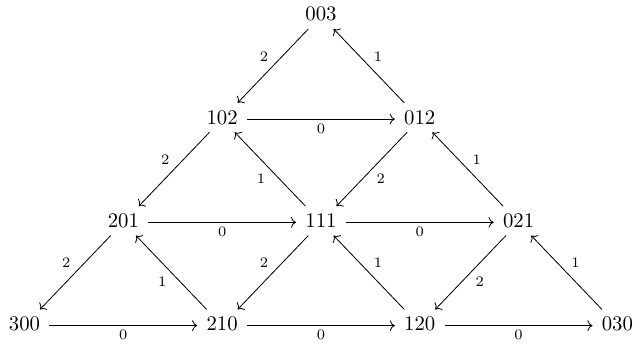}\]

\begin{defn}
For each $s\ge2$, we define a Jacobian algebra $\Pi_\aa^s=\jac(\aa^s,W_\aa^s)$ over $\cc$. The potential is \[W_\aa^s=\sum_{c}\lambda_\aa(c)c,\] where the sum is taken over all $3$-cycles $c=e_x\alpha_{i_0}\alpha_{i_1}\alpha_{i_2}$ in $\aa^s$, and \[\lambda_\aa(c)=
\begin{cases}
    1 & \text{if } (i_0,i_1,i_2)\sim(0,1,2),\\
    -1 & \text{if } (i_0,i_1,i_2)\sim(0,2,1).
\end{cases}\]
Informally, this is the sum of all anti-clockwise 3-cycles minus the sum of all clockwise 3-cycles.
\end{defn}

\begin{remark}
If $\alpha_ie_x$ lies on an edge of $\aa^s$ then $\partial_{\alpha_ie_x}W_\aa^s=e_x\alpha_{i+1}\alpha_{i-1}$, while if $\alpha_ie_x$ is an internal arrow then $\partial_{\alpha_ie_x}W_\aa^s=e_x(\alpha_{i+1}\alpha_{i-1}-\alpha_{i-1}\alpha_{i+1})$. Hence $\Pi_\aa^s$ is the path algebra of $\aa^s$ modulo the relations
\begin{enumerate}
    \item each length two path which starts and ends on the same edge of $\aa^s$ and whose midpoint is not on that edge is zero, 
    \item each rhombus in $\aa^s$ commutes.
\end{enumerate}
Thus $\Pi_\aa^s$ is precisely the algebra called $\widehat{\Lambda}^{(2,s)}$ in \cite[Def 5.1]{iyama2011n}.
\end{remark}

\begin{notation}
\label{Notation: Nakayama automorphism}
Let $s\ge2$. Recall $\omega\colon\nn^3\to\nn^3$, $\omega(x_0,x_1,x_2)=(x_1,x_2,x_0)$. It is clear that $\omega$ permutes $\aa^s_0$, inducing an automorphism of $\aa^s$ such that \[(x\overset{i}\longrightarrow x+f_i)\mapsto(\omega(x)\overset{i-1}\longrightarrow\omega(x)+f_{i-1}).\] This is well-defined since $\omega(x+f_i)=\omega(x)+f_{i-1}$. We abuse notation and write $\omega$ for the induced automorphism of $\cc\aa^s$.
\end{notation}

\begin{prop}
\label{Prop: Pi_A is Frobenius}
\cite[Thm 3.5]{herschend2011n}. For any $s\ge2$, the algebra $\Pi_\aa^s$ is Frobenius, and its Nakayama automorphism is $\omega$.
\end{prop}

Theorem \ref{Theorem: Classification of 2-rep-finite algebras} then implies that, for any $s\ge2$ and cut $C\subset\aa^s_1$, $\Pi_\aa^s/\langle C\rangle$ is 2-representation-finite, and its 3-preprojective algebra is $\Pi_\aa^s$ \cite[Prop 5.48]{iyama2011n}.

Closely related to $d$-representation-finiteness is the fractional Calabi-Yau property. For an algebra $\Lambda$ with $\gldim\Lambda<\infty$, let $\dd=\derb(\mod\Lambda)$ be its bounded derived category. Denote by $\Sigma\colon\dd\to\dd$ the shift functor, and by \[\nu=-\otimes_\Lambda^\mathbf{L}D\Lambda\colon\dd\to\dd\] the derived Nakayama functor, where $D=\hom_\cc(-,\cc)$. Note that $\nu$ is a Serre functor, meaning $\hom_\dd(X,Y)\cong D\hom_\dd(Y,\nu X)$ naturally in $X$ and $Y$ \cite[\S 4.6]{happel1988triangulated}. By \cite[Prop 3.3]{bondal1990representable}, there exists a natural isomorphism $n\colon\nu\Sigma\to\Sigma\nu$ making $(\nu,n)$ a \emph{triangle functor} - for a definition see \cite[\S2.5]{keller2008calabi}.

\begin{defn}
\label{Defn: Fractional Calabi-Yau}
\cite[\S2.6]{keller2008calabi} Let $\Lambda$ be an algebra with $\gldim\Lambda<\infty$. Then $\Lambda$ is \textit{fractional Calabi-Yau} of dimension $N/m$ if there exists an isomorphism of triangle functors \[(\nu,n)^m\cong(\Sigma,-\id_{\Sigma^2})^N\] in $\dd$ for some $N,m\in\zz$, $m\neq0$.
\end{defn}

\begin{remark}
One should treat $N/m$ as a pair of integers, not as a rational number.
\end{remark}

\begin{example}
\cite{dyckerhoff2019simplicial}
For any cut $C\subset\aa^s_1$, $\Pi_\aa^s/\langle C\rangle$ is fractional Calabi-Yau of dimension $2(s-1)/(s+2)$.
\end{example}

\section{Morita equivalence of $\Pi_\dd^s$ and $\Pi_\aa^s\#\zz_3$}

\label{Section: Morita equivalence}

Let $\Lambda$ be an algebra and $G$ be a finite group acting on $\Lambda$. Recall the \textit{skew group algebra} $\Lambda\#G$ is the algebra with underlying vector space $\Lambda\otimes_k kG$, and multiplication defined by \[(a\otimes g)(b\otimes h)=ag(b)\otimes gh\] for $a,b\in\Lambda$ and $g,h\in G$.

In \cite{giovannini2019skew}, the authors consider the case $\Lambda=\jac(Q,W)$ is a Jacobian algebra and $G$ is a finite cyclic group. If certain assumptions are satisfied, they construct a QP $(\widetilde{Q},\widetilde{W})$ such that $\jac(\widetilde{Q},\widetilde{W})$ is Morita equivalent to $\Lambda\#G$. These assumptions are satisfied when $(Q,W)$ is \emph{strongly planar} and $G$ \emph{acts by rotations} \cite[Lem 6.5]{giovannini2019skew}.

\begin{defn}
\cite[Def 8.1]{herschend2011selfinjective} Let $(Q,W)$ be a QP. Its \textit{canvas} $X_{(Q,W)}$ is the 2-dimensional CW complex defined as follows.
\begin{enumerate}
    \item $X_{(Q,W)}^0=Q_0$.
    \item The 1-cells are indexed by $Q_1$, and the attaching maps $\phi^1_\alpha\colon\{0,1\}\to Q^0$ satisfy $\phi_\alpha^1(0)=s(\alpha)$ and $\phi_\alpha^1(1)=t(\alpha)$. Let $\varepsilon_\alpha^1\colon[0,1]\to X_{(Q,W)}$ be a characteristic map extending $\phi_\alpha^1$.
    \item The 2-cells are indexed by cycles $c=\alpha_0\cdots\alpha_{l-1}$ appearing in $W$, and the attaching maps $\phi_c^2\colon S^1\to X_{(Q,W)}^1$ satisfy \[\phi^2_c\left(\cos\left(\frac{2\pi}{l}(i+t)\right),\sin\left(\frac{2\pi}{l}(i+t)\right)\right)=\varepsilon_{\alpha_i}^1(t)\] for integers $0\le i<s$ and real numbers $0\le t< 1$.
\end{enumerate}
Informally, the 1-skeleton is the underlying graph of $Q$, and to obtain the canvas we glue a 2-cell to each cycle appearing in $W$. For general background on CW complexes see e.g. \cite[\S0]{hatcher2002algebraic}.
\end{defn}

\begin{defn}
\cite[Def 6.3]{giovannini2019skew} We call a QP $(Q,W)$ \emph{strongly planar} if there is an embedding $\varepsilon\colon X_{(Q,W)}\to\rr^2$ such that $\im\varepsilon$ is homeomorphic to a disk.
\end{defn}

\begin{lemma}
\label{Lemma: A^s is strongly planar}
For all $s\ge2$, $(\aa^s,W_\aa^s)$ is strongly planar.
\end{lemma}

\begin{proof}
Identify $\rr^2$ with the plane $P=\{x\in\rr^3\mid x_0+x_1+x_2=s-1\}$. Embed the 0-cells (vertices of $\aa^s$) in the obvious way, and embed each $1$-cell $D_{e_x\alpha_i}^1$ as the line segment in $P$ joining $x$ and $x+f_i$. The induced embedding $\varepsilon\colon X_{(\aa^s,W_\aa^s)}\to P$ satisfies $\im\varepsilon=\{x\in P\mid x_0,x_1,x_2\ge 0\}$, which is a closed triangle lying in the plane and therefore homeomorphic to a disk.
\end{proof}

\begin{defn}
\cite[Def 6.4]{giovannini2019skew}
Let $(Q,W)$ be a strongly planar QP, and let $G$ be a finite cyclic group acting on $kQ$. Then $G$ is said to act on $(Q,W)$ \emph{by rotations} if
\begin{enumerate}
    \item there is an embedding $X_{(Q,W)}\to\rr^2$ such that the action of a generator of $G$ is induced by a rotation of the plane;
    \item the action of $G$ is faithful;
    \item every cycle $c$ appearing in $W$ is one of the following types:
    \begin{enumerate}
        \item[(i)] $c$ goes through no vertices fixed by $G$;
        \item[(ii)] $c$ goes through precisely one vertex fixed by $G$ (counted with multiplicity);
        \item[(iii)] $c$ goes through precisely one vertex not fixed by $G$ (counted with multiplicity);
        \item[(iv)] $c$ goes through only vertices fixed by $G$. 
    \end{enumerate}
\end{enumerate}
\end{defn}

\begin{lemma}
\label{Lemma: Z_3 acts by rotations}
For all $s\ge2$, $\zz_3$ acts on $(\aa^s,W_\aa^s)$ by rotations.
\end{lemma}

\begin{proof}
Recall the automorphism $\omega$ of $\cc\aa^s$ from Notation \ref{Notation: Nakayama automorphism}. Clearly $\omega^3=\id$, so there is a group action 
\begin{align*}
    \zz_3&\to\aut(\cc\aa^s),\\
        j&\mapsto\omega^j.
\end{align*}
\begin{enumerate}
    \item Let $\varepsilon\colon X_{(\aa^s,W_\aa^s)}\to P$ be the embedding of Lemma \ref{Lemma: A^s is strongly planar}. Then $\omega$ is induced by rotating the plane $P$ clockwise through $\frac{2\pi}{3}$, about the point $(\frac{s-1}{3},\frac{s-1}{3},\frac{s-1}{3})$.
    \item Since $s\ge2$, $(s-1,0,0)$ is a vertex of $\aa^s$ not fixed by $\omega$ or $\omega^2$, so the action of $\zz_3$ is faithful.
    \item A vertex $x=(x_0,x_1,x_2)\in\aa^s_0$ is fixed by $\zz_3$ if and only if $x_0=x_1=x_2$. Since $x_0+x_1+x_2=s-1$, this occurs if and only if $s=3t+1$ for some $t\in\zz^+$, in which case there is a unique fixed vertex $(t,t,t)$. Let $c$ be a cycle appearing in $W_\aa^s$. Then $c$ passes through any given vertex of $\aa^s$ at most once. Since there is at most one vertex fixed by $\zz_3$, $c$ is either type (i) or type (ii).
\end{enumerate}
\end{proof}

\begin{theorem}
\cite[Thm 3.20]{giovannini2019skew}
Let $(Q,W)$ be a strongly planar QP and $G$ be a finite cyclic group acting by rotations. Then $G$ acts on $\jac(Q,W)$, and there is an explicit construction of a QP $(\widetilde{Q},\widetilde{W})$ and an idempotent $\eta\in\jac(Q,W)\#G$ such that $\jac(\widetilde{Q},\widetilde{W})\cong\eta(\jac(Q,W)\#G)\eta$.
\end{theorem}

In particular, there exists a QP $(\widetilde{\aa^s},\widetilde{W_\aa^s})$ such that $\jac(\widetilde{\aa^s},\widetilde{W_\aa^s})$ is Morita equivalent to $\Pi_\aa^s\#\zz_3$. We will describe the construction in this case, under the headings `Vertices', `Arrows' and `Potential', and conclude that in fact $(\widetilde{\aa^s},\widetilde{W_\aa^s})=(\dd^s,W_\dd^s)$.

\subsection*{Vertices}

Let $V_1=\{x\in\aa^s_0\mid x\succ\omega(x),\ x\succ\omega^2(x)\}$. Writing $X=(\frac{s-1}{3},\frac{s-1}{3},\frac{s-1}{3})$, let $V_2=\{X\}$ if $s\equiv1$ (mod 3), and $V_2=\emptyset$ otherwise. Now $V_1\sqcup V_2$ is a complete set of representatives of the $\zz_3$-orbits of vertices of $\aa^s$, and $V_2$ contains precisely the vertices fixed by $\zz_3$. By \cite[Notation 3.9-11]{giovannini2019skew}, $\widetilde{\aa^s_0}=V_1$ if $s\not\equiv1$ (mod 3), and $\widetilde{\aa^s_0}=V_1\cup\{X_k\mid k=0,1,2\}$ if $s\equiv1$ (mod 3).

\subsection*{Arrows}

For each arrow in $\aa^s_1$, we fix a representative of its $\zz_3$-orbit, and define arrows(s) in $\widetilde{\aa^s_1}$ corresponding to the representatives \cite[Notation 3.13]{giovannini2019skew}. There are three cases.
\begin{enumerate}
    \item Consider an arrow in $\aa^s_1$ between two vertices not fixed by $\zz_3$. There is a unique arrow in its orbit whose target is in $\widetilde{\aa^s_0}$, and it must be of the form $\alpha_i\colon\omega^j(x)\longrightarrow \omega^j(x)+f_i$ for some $x\in\widetilde{\aa^s_0}$, $i,j\in\{0,1,2\}$. We define a corresponding arrow $\alpha_{i,j}\colon x\longrightarrow\omega^j(x)+f_i$ in $\widetilde{\aa^s_1}$.
    \item Suppose $s\equiv 1$ (mod 3). There are three arrows in $\aa^s_1$ whose target is $X$, all in the same $\zz_3$-orbit. Only one of them has its source in $\widetilde{\aa^s_0}$, namely $\alpha_0\colon X-f_0\longrightarrow X$. We define three corresponding arrows $\{\beta_k\colon X-f_0\longrightarrow X_k\mid k=0,1,2\}$ in $\widetilde{\aa^s_1}$.
    \item Suppose $s\equiv1$ (mod 3). There are three arrows in $\aa^s_1$ whose source is $X$, all in the same $\zz_3$-orbit. Only one of them has its target in $\widetilde{\aa^s_0}$, namely $\alpha_2\colon X\longrightarrow X+f_2$. We define three corresponding arrows $\{\gamma_k\colon X_k\longrightarrow X+f_2\mid k=0,1,2\}$ in $\widetilde{\aa^s_1}$.
\end{enumerate}

\subsection*{Potential}

We fix a complete set $C_1\sqcup C_2$ of representatives of the $\zz_3$-orbits of 3-cycles in $\aa^s$, where $C_1$ contains cycles that do not pass through the fixed vertex $X$, and $C_2$ contains cycles that do. To each $c\in C_1$ we associate a cycle $\widetilde{c}$ in $\widetilde{\aa^s}$, and to each $c\in C_2$ we associate three cycles $\widetilde{c_0},\widetilde{c_1},\widetilde{c_2}$ in $\widetilde{\aa^s}$ \cite[Notation 3.17]{giovannini2019skew}.
\begin{enumerate}
    \item Consider a 3-cycle in $\aa^s$ that does not pass through $X$. Choose a representative $c$ of its $\zz_3$-orbit that passes through at least one vertex in $\widetilde{\aa^s_0}$. Then $c$ is of the form \[x\overset{l_0}\longrightarrow y'\overset{l_1}\longrightarrow z'\overset{l_2}\longrightarrow x\] where $x\in\widetilde{\aa^s_0}$, $\{l_0,l_1,l_2\}=\{0,1,2\}$, $y'=x+f_{l_0}$ and $z'=x-f_{l_2}$. There exist unique $j_0,j_2\in\{0,1,2\}$ such that $y\coloneqq\omega^{j_0}(y')\in\widetilde{\aa^s_0}$ and $z\coloneqq\omega^{-j_2}(z')\in\widetilde{\aa^s_0}$. 
    
    To simplify notation moving forward, write $i_0=l_0-j_0$, $i_1=l_1+j_2$, $i_2=l_2$ and $j_1=-(j_0+j_2)$. Define $\widetilde{c}$ to be the cycle \[x\overset{(i_0,j_0)}\longrightarrow y\overset{(i_1,j_1)}\longrightarrow z\overset{(i_2,j_2)}\longrightarrow x\] in $\widetilde{\aa^s}$. For a proof $\widetilde{c}$ that exists as claimed, see Remark \ref{Remark: Justification}.
    \item Suppose $s\equiv1$ (mod 3). There are six cycles in $W_\aa^s$ passing through $X$, in two disjoint $\zz_3$-orbits. As representatives of these orbits we choose
    \begin{align*}
        c^-&:\;\;\;X+f_2\overset{1}\longrightarrow X-f_0\overset{0}\longrightarrow X\overset{2}\longrightarrow X+f_2,\\c^+&:\;\;\;
        X+f_2\overset{0}\longrightarrow X-f_1\overset{1}\longrightarrow X\overset{2}\longrightarrow X+f_2.
    \end{align*}
    For each $k\in\{0,1,2\}$, we define the cycles
    \begin{align*}
        \widetilde{c^-_k}&:\;\;\;X+f_2\overset{1}\longrightarrow X-f_0\overset{\beta_k}\longrightarrow X_k\overset{\gamma_k}\longrightarrow X+f_2,\\
        \widetilde{c^+_k}&:\;\;\;X+f_2\overset{(2,1)}\longrightarrow X-f_0\overset{\beta_k}\longrightarrow X_k\overset{\gamma_k}\longrightarrow X+f_2
    \end{align*}
    in $\widetilde{\aa^s}$. We also define $p(c^-)=0$, $p(c^+)=-1$. Informally, this will adjust for the fact $X-f_1\not\in\widetilde{\aa^s_0}$ but $\omega(X-f_1)\in\widetilde{\aa^s_0}$.
\end{enumerate}
By \cite[Notation 3.18]{giovannini2019skew}, the potential on $\widetilde{\aa^s}$ is given by \[\widetilde{W_\aa^s}=\sum_{c\in C_1}\lambda_\aa(c)\frac{|\zz_3\cdot c|}{3}\widetilde{c}+\sum_{c\in C_2}\lambda_\aa(c)\sum_{k=0}^2\zeta^{-p(c)k}\widetilde{c}_k.\]

\begin{remark}
\label{Remark: Justification}
We show the cycle $\widetilde{c}$ defined above exists. Recall $c$ is the cycle \[x\overset{l_0}\longrightarrow\omega^{-j_0}(y)\overset{l_1}\longrightarrow\omega^{j_2}(z)\overset{l_2}\longrightarrow x\] in $\aa^s$, where $x,y,z\in\widetilde{\aa^s_0}$, $\{l_0,l_1,l_2\}=\{0,1,2\}$ and $j_0,j_2\in\{0,1,2\}$.

We find the arrows in $\widetilde{\aa^s}$ induced by each arrow in $c$. Note that $\omega^j(e_x\alpha_i)=e_{\omega^j(x)}\alpha_{i-j}$ for all $i,j\in\{0,1,2\}$ and $x\in\aa^s_0$. Recall we write $i_0=l_0-j_0$, $i_1=l_1+j_2$, $i_2=l_2$ and $j_1=-(j_0+j_2)$.
\begin{enumerate}
    \item Consider $x\overset{l_0}\longrightarrow\omega^{-j_0}(y)$. Applying $\omega^{j_0}$, we see that the arrow in its $\zz_3$-orbit with target in $\widetilde{\aa^s_0}$ is $(\omega^{j_0}(x)\overset{l_0-j_0}\longrightarrow y)=(\omega^{j_0}(x)\overset{i_0}\longrightarrow y)$. Hence, there is an arrow $x\overset{(i_0,j_0)}\longrightarrow y$ in $\widetilde{\aa^s}$.
    \item Consider $\omega^{-j_0}(y)\overset{l_1}\longrightarrow\omega^{j_2}(z)$. Applying $\omega^{-j_2}$, we see that the arrow in its $\zz_3$-orbit with target in $\widetilde{\aa^s_0}$ is $(\omega^{-(j_0+j_2)}(y)\overset{l_1+j_2}\longrightarrow z)=(\omega^{j_1}(y)\overset{i_1}\longrightarrow z)$. Hence, there is an arrow $y\overset{(i_1,j_1)}\longrightarrow z$ in $\widetilde{\aa^s}$.
    \item Note that $\omega^{j_2}(z)\overset{l_2}\longrightarrow x$ already has target in $\widetilde{\aa^s_0}$ so, using $i_2=l_2$, there is an arrow $z\overset{(i_2,j_2)}\longrightarrow x$ in $\widetilde{\aa^s_0}$.
\end{enumerate}

Hence the cycle $\widetilde{c}$ given by $x\overset{(i_0,j_0)}\longrightarrow y\overset{(i_1,j_1)}\longrightarrow z\overset{(i_2,j_2)}\longrightarrow x$ exists in $\widetilde{\aa^s}$ as claimed.
\end{remark}

We now prove Theorem 1. The reader may wish to recall Definitions \ref{Defn: D^s} and \ref{Defn: Potential on D^s}.

\begin{theorem}
\label{Theorem 1 body}
For all $s\ge2$, $(\widetilde{\aa^s},\widetilde{W_\aa^s})=(\dd^s,W_\dd^s)$. Hence, $\Pi_\dd^s\cong\eta(\Pi_\aa^s\#\zz_3)\eta$ for some idempotent $\eta\in\Pi_\aa^s\#\zz_3$.
\end{theorem}

\begin{proof}
By inspection, we have that $\widetilde{\aa^s_0}=\dd^s_0$ and $\widetilde{\aa^s_1}\subseteq\dd^s_1$. To see that $\dd^s_1\subseteq\widetilde{\aa^s_1}$, let $e_x\alpha_{i,j}\in\dd^s_1$. Then $x,\omega^j(x)+f_i\in\widetilde{\aa^s_0}$, so certainly $\omega^j(x)\in\aa^s_0$. Thus $e_{\omega^j(x)}\alpha_i$ is an arrow in $\aa^s_1$ between two vertices not fixed by $\zz_3$, whose target is in $\widetilde{\aa^s_0}$. Hence $e_x\alpha_{i,j}\in \widetilde{\aa^s_1}$ by construction, and $\widetilde{\aa^s}=\dd^s$.

Note that both $\widetilde{W_\aa^s}$ and $W_\dd^s$ are linear combinations of 3-cycles in $\dd^s$. So to show $\widetilde{W_\aa^s}=W_\dd^s$ it is enough to check, for every 3-cycle $c$ in $\dd^s$, that its coefficient $\widetilde{\lambda_\aa}(c)$ in $\widetilde{W_\aa^s}$ is equal to $\lambda_\dd(c)$.
\begin{enumerate}
    \item Suppose $c=e_x\alpha_{i_0,j_0}\alpha_{i_1,j_1}\alpha_{i_2,j_2}$ is the product of three distinct arrows. If $j_0+j_1+j_2\neq0$, then $c$ is not in $\widetilde{W_\aa^s}$ by construction, so $\widetilde{\lambda_\aa}(c)=0=\lambda_\dd(c)$. Otherwise, $c=\widetilde{d}$, where $d$ is the cycle $e_x\alpha_{i_0+j_0}\alpha_{i_1-j_2}\alpha_{i_2}$ in $\aa^s$. Hence
    \begin{align*}
        \widetilde{\lambda_\aa}(c)&=\lambda_\aa(d)\frac{|\zz_3\cdot d|}{3}\\
        &=\lambda_\aa(d)\\
        &=
        \begin{cases}
            1 & \text{if } (i_0+j_0,i_1-j_2,i_2)\sim(0,1,2),\\
            -1 & \text{if } (i_0+j_0,i_1-j_2,i_2)\sim(0,2,1)
        \end{cases}\\
        &=\lambda_\dd(c).
    \end{align*}
    \item Suppose $c=e_x\alpha_{i,j}^3$ for a loop $e_x\alpha_{i,j}$. If $3j\neq 0$ then $c$ is not in $\widetilde{W_\aa^s}$ by construction, so $\widetilde{\lambda_\aa}(c)=0=\lambda_\dd(c)$. Otherwise, $c=\widetilde{d}$, where $d$ is the cycle $e_x\alpha_{i+j}\alpha_{i-j}\alpha_{i}$ in $\aa^s$. Hence
    \begin{align*}
        \widetilde{\lambda_\aa}(c)&=\lambda_\aa(d)\frac{|\zz_3\cdot d|}{3}\\
        &=\frac{\lambda_\aa(d)}{3}\\
        &=
        \begin{cases}
            \frac{1}{3} & \text{if } (i+j,i-j,i)\sim(0,1,2),\\
            -\frac{1}{3} & \text{if } (i+j,i-j,i)\sim(0,1,2)
        \end{cases}\\
        &=\lambda_\dd(c).
    \end{align*}
    \item Suppose $s\equiv1$ (mod 3) and let $c$ be a $3$-cycle in $\dd^s$ passing through $X_k$ for some $k\in\{0,1,2\}$. Then either $c=\alpha_1\beta_k\gamma_k=\widetilde{c^-_k}$ or $c=\alpha_{2,1}\beta_k\gamma_k=\widetilde{c^+_k}$. Now
    \begin{gather*}
        \widetilde{\lambda_\aa}(\alpha_1\beta_k\gamma_k)=\lambda_\aa(c^-)\zeta^{-p(c^-)k}=-1=\lambda_\dd(\alpha_1\beta_k\gamma_k),\\
        \widetilde{\lambda_\aa}(\alpha_{2,1}\beta_k\gamma_k)=\lambda_\aa(c^+)\zeta^{-p(c^+)k}=\zeta^k=\lambda_\dd(\alpha_{2,1}\beta_k\gamma_k)
    \end{gather*}
    and we are done.
\end{enumerate}
\end{proof}

\section{Connection with operator algebras}

\label{Section: Operator algebras}

In \cite{evans2009ocneanu,evans2012nakayama}, Evans and Pugh study Jacobian algebras of the quivers $\aa^s$ and $\dd^s$ (considered as the SU(3) $\aa\dd\ee$ graphs of Di Francesco and Zuber \cite{di1990n}) with respect to different potentials. We show that their algebras are  isomorphic to those considered in this article. Note that the quiver we call $\aa^s$ they call $\aa^{(s+2)}$, and likewise $\dd^s$ corresponds to $\dd^{(s+2)}$.

\subsection*{Type A}

In the following we write $q=e^{\pi i/(s+2)}$, where $s\ge2$ is the index of the quiver $\aa^s$. For each $n\in\zz^+$, we define the \emph{quantum number} $[n]=(q^n-q^{-n})/(q-q^{-1})$.

\begin{defn}
\cite[Thm 5.1]{evans2009ocneanu} For each $s\ge2$, define a potential on $\aa^s$ by \[V_\aa^s=\sum_c\mu_\aa(c)c,\] where the sum runs over all 3-cycles in $\aa^s$ and, for $x=(x_0,x_1,x_2)\in\aa^s_0$,
\begin{align*}
    \mu_\aa(e_x\alpha_0\alpha_1\alpha_2)&=\sqrt{[x_1+1][x_1+2][x_2+1][x_2+2][x_1+x_2+2][x_1+x_2+3]}/[2],\\
    \mu_\aa(e_x\alpha_0\alpha_2\alpha_1)&=\sqrt{[x_1+1][x_1+2][x_2][x_2+1][x_1+x_2+2][x_1+x_2+3]}/[2].
\end{align*}
\end{defn}

\begin{remark}
We believe there is a small typo in \cite[Thm 5.1]{evans2009ocneanu}. Namely, the formula (14) should have $[k+m+3][k+m+4]$ in place of $[k+m+2][k+m+3]$. Translating into our notation gives the above definition.
\end{remark}

The following lemma will be the key tool in proving $\Pi_\aa^s\cong\jac(\aa^s,V_\aa^s)$.

\begin{lemma}
\label{Lemma: Induced isomorphism of Jacobian algebras}
\cite[Prop 3.7]{derksen2008quivers} Let $(Q,W)$ be a QP and let $Q'$ be a quiver. Any algebra isomorphism $f\colon kQ\to kQ'$ induces an algebra isomorphism $\jac(Q,W)\cong\jac(Q',f(W))$.
\end{lemma}

We first show that one can replace the commutativity relations in $\Pi_\aa^s$ with anti-commutativity relations, inspired by work in \cite[\S3.3]{grant2019higher}.

\begin{lemma}
\label{Lemma: Anticommutativity Type A}
Let $s\ge2$. Denote by $|W_\aa^s|$ the potential on $\aa^s$ given by the sum of all 3-cycles, each with coefficient 1. Then $\Pi_\aa^s\cong\jac(\aa^s,|W_\aa^s|)$.
\end{lemma}

\begin{proof}
Let $\parity_i(x)=(-1)^{s-x_{i+1}}$ for all $x=(x_0,x_1,x_2)\in\aa^s_0$ and $i\in\{0,1,2\}$. Let $\varphi\colon\cc\aa^s\to\cc\aa^s$ be the algebra automorphism induced by $e_x\alpha_i\mapsto\parity_i(x)e_x\alpha_i$. Consider a cycle $c$ in $W_\aa^s$. If $c=e_x\alpha_0\alpha_1\alpha_2$ then $\lambda_\aa(c)=1$ so
\begin{align*}
    \varphi(\lambda_\aa(c)c)&=\parity_0(x)\parity_1(x+f_0)\parity_2(x-f_2)c\\
    &=(-1)^{3s+1-(x_0+x_1+x_2)}c\\
    &=(-1)^{2s+2}c\\
    &=c,
\end{align*}
where we used that $x_0+x_1+x_2=s-1$. Similarly if $c=e_x\alpha_0\alpha_2\alpha_1$ then one can check $\varphi(\lambda_\aa(c)c)=c$.
Hence $\varphi(W_\aa^s)=|W_\aa^s|$, so $\Pi_\aa^s\cong\jac(\aa^s,|W_\aa^s|)$ by Lemma \ref{Lemma: Induced isomorphism of Jacobian algebras}.
\end{proof}

We can now make our conclusion.

\begin{prop}
\label{Prop: EP Type A}
For all $s\ge2$, $\Pi_\aa^s\cong\jac(\aa^s,V_\aa^s)$.
\end{prop}

\begin{proof}
We show $\jac(\aa^s,|W_\aa^s|)\cong\jac(\aa^s,V_\aa^s)$, at which point the statement follows by Lemma \ref{Lemma: Anticommutativity Type A}. Let
\begin{align*}
    \cof_0(x)&=\sqrt[4]{[x_1+1][x_1+2][x_1+x_2+2][x_1+x_2+3]}/\sqrt[3]{[2]},\\
    \cof_1(x)&=\sqrt[4]{[x_1][x_1+1][x_2+1][x_2+2]}/\sqrt[3]{[2]},\\
    \cof_2(x)&=\sqrt[4]{[x_2][x_2+1][x_1+x_2+1][x_1+x_2+2]}/\sqrt[3]{[2]}
\end{align*}
for all $x=(x_0,x_1,x_2)\in\aa^s_0$. Let $\psi\colon\cc\aa^s\to\cc\aa^s$ be the algebra automorphism induced by $e_x\alpha_i\mapsto\cof_i(x)e_x\alpha_i$. In view of Lemma \ref{Lemma: Induced isomorphism of Jacobian algebras}, it is enough to show that $\psi(|W_\aa^s|)=V_\aa^s$. Indeed, it is straightforward to calculate
\begin{align*}
    \psi(e_x\alpha_0\alpha_1\alpha_2)&=\cof_0(x)\cof_1(x+f_0)\cof_2(x-f_2)e_x\alpha_0\alpha_1\alpha_2=\mu_\aa(e_x\alpha_0\alpha_1\alpha_2)e_x\alpha_0\alpha_1\alpha_2,\\
    \psi(e_x\alpha_0\alpha_2\alpha_1)&=\cof_0(x)\cof_2(x+f_0)\cof_1(x-f_1)e_x\alpha_0\alpha_2\alpha_1=\mu_\aa(e_x\alpha_0\alpha_2\alpha_1)e_x\alpha_0\alpha_2\alpha_1,
\end{align*}
so we are done.
\end{proof}

\subsection*{Type D}

To get from $(\aa^s,W_\aa^s)$ to $(\aa^s,V_\aa^s)$, the only change we make is to multiply the coefficient of each cycle in the potential by a non-zero constant. The canvas does not change, i.e. $X_{(\aa^s,W_\aa^s)}=X_{(\aa^s,V_\aa^s)}$. In particular, $(\aa^s,V_\aa^s)$ is strongly planar and $\zz_3$ acts by rotations. Thus we can apply the construction of \cite[\S3.2-3]{giovannini2019skew} to $(\aa^s,V_\aa^s)$. The specific coefficients in the potential do not play a role in the construction of the quiver, nor of the idempotent. Hence we obtain a potential $\widetilde{V_\aa^s}$ such that $\jac(\dd^s,\widetilde{V_\aa^s})\cong\eta(\jac(\aa^s,V_\aa^s)\#\zz_3)\eta$, where $\eta$ is as in Theorem \ref{Theorem 1 body}. Furthermore, a cycle appears in $W_\aa^s$ if and only if it appears in $V_\aa^s$, so the same cycles appear in $\widetilde{W_\aa^s}$ and $\widetilde{V_\aa^s}$. That is to say,
\[\widetilde{V_\aa^s}=\sum_{c\in C_1}\mu_\aa(c)\frac{|\zz_3\cdot c|}{3}\widetilde{c}+\sum_{c\in C_2}\mu_\aa(c)\sum_{k=0}^2\zeta^{-p(c)k}\widetilde{c}_k.\]

\begin{defn}
\label{Defn: EP Type D}
\cite[Thm 6.1, 6.2]{evans2009ocneanu} For each $s\ge2$, define a potential on $\dd^s$ by \[V_\dd^s=\sum_c\mu_\dd(c)c,\] where the sum runs over all 3-cycles in $\dd^s$, and $\mu_\dd$ is defined as follows.

If $s=3t+1$ for some $t\in\zz^+$, then
\begin{align*}
    \mu_\dd(\alpha_1\beta_k\gamma_k)&=\zeta^k\frac{[t]\sqrt{[t+1]^3[t+2]}}{\sqrt{3}[2]},\\
    \mu_\dd(\alpha_{2,1}\beta_k\gamma_k)&=\zeta^{-k}\frac{[t+2]\sqrt{[t][t+1]^3}}{\sqrt{3}[2]}.
\end{align*}
For all other cycles $c$, including when $s\not\equiv1$ (mod 3), $\mu_\dd(c)=\widetilde{\mu_\aa}(c)$.
\end{defn}

\begin{remark}
In \cite[Thm 6.2]{evans2009ocneanu}, the authors explicitly give the coefficients of six cycles, and say the rest are given by the ``corresponding" cycles in $\aa^s$. Hence, we set $\mu_\dd(c)=\widetilde{\mu_\aa}(c)$ for all other cycles $c$. Once we make this assumption, we already get the correct coefficients for four of the distinguished cycles, so we omit them in the above definition.
\end{remark}

\begin{prop}
Let $s\ge2$. Then $\Pi_\dd^s\cong\jac(\dd^s,V_\dd^s)$. 
\end{prop}

\begin{proof}
First, we show that $\jac(\dd^s,\widetilde{V_\aa^s})\cong\jac(\dd^s,V_\dd^s)$. To this end, let $\chi\colon\cc\dd^s\to\cc\dd^s$ be the algebra automorphism such that $\chi(\beta_k)=\frac{\zeta^k}{\sqrt{3}}\beta_k$ and $\chi$ is the identity on all other arrows. We check that $\chi(\widetilde{\mu_\aa}(c)c)=\mu_\dd(c)c$ for all 3-cycles $c$ in $\dd^s$, whence $\chi(\widetilde{V_\aa^s})=V_\dd^s$ and the claim follows by Lemma \ref{Lemma: Induced isomorphism of Jacobian algebras}.

If $c$ is not one of the distinguished cycles in Definition \ref{Defn: EP Type D}, $\chi(\widetilde{\mu_\aa}(c)c)=\widetilde{\mu_\aa}(c)c=\mu_\dd(c)c$ by definition.

If $c=\alpha_1\beta_k\gamma_k$ then $c=\widetilde{c_k^-}$, where $c^-=e_{(t+1,t-1,t)}\alpha_0\alpha_2\alpha_1$ (see the construction of the potential in \S\ref{Section: Morita equivalence}). So
\begin{align*}
    \chi(\widetilde{\mu_\aa}(c)c)&=\frac{\zeta^k}{\sqrt{3}}\widetilde{\mu_\aa}(c)c\\
    &=\frac{\zeta^k}{\sqrt{3}}\mu_\aa(c^-)\zeta^{-p(c^-)k}c\\
    &=\zeta^k\frac{[t+1]\sqrt{[t][t+2][2t+1][2t+2]}}{\sqrt{3}[2]}c\\
    &=\zeta^k\frac{[t]\sqrt{[t+1]^3[t+2]}}{\sqrt{3}[2]}c\\
    &=\mu_\dd(c)c,
\end{align*}
where we used that $p(c^-)=0$ and that $[n]=[3t+3-n]$ for all $1\le n\le 2t+2$ \cite[Lem 4.3]{evans2009ocneanu}.

Similarly, if $c=\alpha_{2,1}\beta_k\gamma_k$ then $c=\widetilde{c_k^+}$, where $c^+=e_{(t+1,t,t-1)}\alpha_0\alpha_1\alpha_2$. Recalling $p(c^+)=-1$, one calculates as above to show $\chi(\widetilde{\mu_\aa}(c)c)=\mu_\dd(c)c$. Therefore $\jac(\dd^s,\widetilde{V_\aa^s})\cong\jac(\dd^s,V_\dd^s)$.

By Proposition \ref{Prop: EP Type A}, we have an isomorphism $\psi\colon\Pi_\aa^s\to\jac(\aa^s,V_\aa^s)$, which induces an isomorphism $\psi\otimes\id\colon\Pi_\aa^s\#\zz_3\to\jac(\aa^s,V_\aa^s)\#\zz_3$. Note that $\eta=\sum d_i\otimes e_j$ for some idempotents $d_i\in\Pi_\aa^s$ (which are length zero) and $e_j\in\cc\zz_3$ \cite[Notation 3.11]{giovannini2019skew}. Since $\psi$ acts trivially on the $d_i$, $(\psi\otimes\id)(\eta)=\eta$, and we have a chain of isomorphisms \[\Pi_\dd^s\cong\eta(\Pi_\aa^s\#\zz_3)\eta\overset{\psi\otimes\id}\longrightarrow\eta(\jac(\aa^s,V_\aa^s)\#\zz_3)\eta\cong\jac(\dd^s,\widetilde{V_\aa^s})\overset\chi\longrightarrow\jac(\dd^s,V_\dd^s).\]
\end{proof}

\section{Cuts of $\Pi_\dd^s$}

\label{Section: Cuts}

We would like to emulate type A, and take cuts of $\Pi_\dd^s$ to obtain some 2-representation-finite algebras of type D. This turns out to be impossible if $s\not\equiv1$ (mod 3).

\begin{lemma}
Let $s\ge2$, $s\not\equiv1$ (mod 3). There are zero cuts of $(\dd^s,W_\dd^s)$.
\end{lemma}

\begin{proof}
If $s=3t$ for some $t\in\zz^+$, there is a loop $\alpha_{2,2}\colon x\longrightarrow x$ in $\dd^s_1$, where $x=(t,t,t-1)$. By Definition \ref{Defn: Potential on D^s}, the 3-cycle $e_x\alpha_{2,2}^3$ appears in $W_\dd^s$ with non-zero coefficient. For any subset $C\subseteq\dd^s_1$, either $g_C(e_x\alpha_{2,2}^3)=0$ or $g_C(e_x\alpha_{2,2}^3)=3$, so $C$ is not a cut of $(\dd^s,W_\dd^s).$

If $s=3t+2$ for some $t\in\nn$, there is a loop $\alpha_{2,1}\colon y\longrightarrow y$ in $\dd^s_1$, where $y=(t+1,t,t)$. As before, $e_y\alpha_{2,1}^3$ appears in $W_\dd^s$ with non-zero coefficient, and the same argument shows there are no cuts of $(\dd^s,W_\dd^s)$.
\end{proof}

Thus, it is impossible for $\Pi_\dd^s$ to be a 3-preprojective algebra if $s\not\equiv1$ (mod 3). We will see that if $s\equiv1$ (mod 3), then $\Pi_\dd^s$ is indeed a 3-preprojective algebra.

\begin{defn}
\cite[Def 7.5]{giovannini2019skew} Let $(Q,W)$ be a $QP$, and let $G$ be a group acting on $\jac(Q,W)$. We call a cut $C\subseteq Q_1$ \textit{G-invariant} if $g\cdot\alpha\in C$ for all $\alpha\in C$. We say a QP $(Q,W)$ has \emph{enough ($G$-invariant) cuts} if, for every $\alpha\in Q_1$, there exists a ($G$-invariant) cut $C$ with $\alpha\in C$.
\end{defn}

\begin{lemma}
\label{Lemma: D^s has enough cuts}
Let $s=3t+1$ for some $t\in\zz^+$. Then $(\dd^s,W_\dd^s)$ has enough cuts.
\end{lemma}

\begin{proof}
Consider an arrow $\delta\in\dd^s_1$. There are three possibilities.
\begin{enumerate}
    \item $\delta=e_x\alpha_{i,j}$ for some $x\in\dd^s_0$, $i,j\in\{0,1,2\}$. In this case let $\delta'=e_{\omega^j(x)}\alpha_i\in\aa^s_1$.
    \item $\delta=\beta_k$ for some $k\in\{0,1,2\}$. In this case let $\delta'=\alpha_0e_X\in\aa^s_1$, where $X=(t,t,t)$.
    \item $\delta=\gamma_k$ for some $k\in\{0,1,2\}$. In this case let $\delta'=e_X\alpha_2\in\aa^s_1$.
    \end{enumerate}
By \cite[Prop 8.2]{giovannini2019skew}, $(\aa^s,W_\aa^s)$ has enough $\zz_3$-invariant cuts, so there exists a $\zz_3$-invariant cut $C'\subset\aa^s_1$ containing $\delta'$. By the recipe in \cite[Prop 7.3]{giovannini2019skew},
\begin{equation}
\label{Equation: Induced cut}
    C=\{e_x\alpha_{i,j}\mid x\in Q_0,\ e_{\omega^j(x)}\alpha_i\in C'\}\cup\{\beta_0,\beta_1,\beta_2\mid\alpha_0e_X\in C'\}\cup\{\gamma_0,\gamma_1,\gamma_2\mid e_X\alpha_2\in C'\}
\end{equation} is a cut of $(\dd^s,W_\dd^s)$. Since $\delta'\in C'$, we have $\delta\in C$, so $(\dd^s,W_\dd^s)$ has enough cuts.
\end{proof}

\begin{lemma}
\label{Lemma: Skew group algebra is symmetric}
\cite[Cor 2.6]{giovannini2019skew}
Let $\Lambda$ be a Frobenius algebra, whose Nakayama automorphism $\sigma$ generates $\im(G)\subseteq\aut(\Lambda)$. Then $\Lambda\#G$ is symmetric.
\end{lemma}

We now prove Theorem \ref{Theorem 2}.

\begin{theorem}
\label{Theorem 2 body}
For all $s\ge2$, $\Pi_\dd^s$ is symmetric. If $s\equiv1$ (mod 3) and $C\subset\dd^s_1$ is a cut, then $\Pi_\dd^s/\langle C\rangle$ is 2-representation-finite, and its 3-preprojective algebra is $\Pi_\dd^s$.
\end{theorem}

\begin{proof}
Proposition \ref{Prop: Pi_A is Frobenius} says that $\Pi_\aa^s$ is Frobenius. Its Nakayama automorphism $\omega$ generates $\im(\zz_3)\subseteq\aut(\Pi_\aa^s)$, so we can apply Lemma \ref{Lemma: Skew group algebra is symmetric} to conclude that $\Pi_\aa^s\#\zz_3$ is symmetric. Being symmetric is Morita invariant, so $\Pi_\dd^s$ is also symmetric. In particular, $(\dd^s,W_\dd^s)$ is a selfinjective QP, so if $s\equiv1$ (mod 3) and $C\subset\dd^s_1$ is a cut, then $\Pi_\dd^s/\langle C\rangle$ is 2-representation-finite by Theorem \ref{Theorem: Classification of 2-rep-finite algebras}.
\end{proof}

We can say more.

\begin{prop}
\cite[Thm 7.9]{giovannini2019skew}
Let $(Q,W)$ be a strongly planar selfinjective QP, with a group $G$ acting by rotations and enough $G$-invariant cuts. Then all truncated Jacobian algebras $\jac(\widetilde{Q},\widetilde{W})_C$ are derived equivalent.
\end{prop}

\begin{corollary}
\label{Corollary: Truncations are derived equivalent}
Let $s\equiv1$ (mod 3). All truncations $\Pi^s_\dd/\langle C\rangle$ are derived equivalent.
\end{corollary}

\section{$\Pi_\dd^s/\langle C\rangle$ is fractional Calabi-Yau}

\label{Section: Calabi-Yau}

We use a theorem of Grant, which relates the Nakayama automorphism of a $(d+1)$-preprojective algebra to the fractional Calabi-Yau property of the associated $d$-representation-finite algebras.

Throughout, let $\Lambda=\jac(Q,W)_C$ be a basic 2-representation-finite algebra, and let $\sigma$ be the Nakayama automorphism of $\Pi=\jac(Q,W)$.

\begin{lemma}\cite[Prop 1.3]{iyama2011cluster}, \cite[Prop 3.2]{grant2020nakayama}.
\label{Lemma: l function}
\begin{itemize}
	\item[(a)]There exists a function $l\colon Q_0\to\zz$ such that $D(\Lambda e_i)\cong\tau_2^{-l(i)}(\sigma(e_i)\Lambda)$ for all $i\in Q_0$.
	\item[(b)]For all $i\in Q_0$, $e_i\Pi\cong D(\Pi\sigma(e_i))\{-l(i)\}$ as graded $\Pi$-modules.
\end{itemize}
Here $\{-\}$ denotes a grading shift, i.e. $(M\{n\})_r=M_{r+n}$.
\end{lemma}

Our main tool will be the following theorem.

\begin{theorem}
\label{Theorem: FInite order Nakayama implies fCY}
If there exists $k\in\zz^+$ such that $\sigma^k=\id$, then 
\begin{itemize}
    \item[(a)] there exists $N\in\zz$ such that $\sum_{j=1}^k l(\sigma^j(i))=N$ for all $i\in Q_0$;
    \item[(b)] $\Lambda$ is fractional Calabi-Yau of dimension $2N/(k+N)$.
\end{itemize}
\end{theorem}

\begin{proof}
In \cite[\S3.4]{grant2020nakayama}, it is demonstrated that if $f\in\hom_\Lambda(e_i\Lambda,\tau_2^{-r}(e_j\Lambda))$ then $\sigma(f)\in\hom_\Lambda(\sigma(e_i)\Lambda,\tau_2^{l(i)-l(j)-r}(\sigma(e_j)\Lambda))$. In other words, if $f\in\Pi_r$ (the degree $r$ part of $\Pi$) and $i,j\in Q_0$, then $e_jfe_i\in\Pi_{r+l(i)-l(j)}$. This means $(\sigma,l)$ is a \emph{degree adjusted automorphism} of $\Pi$, so \cite[Lem 6.15]{grant2022serre} gives us (a). By Lemma \ref{Lemma: l function}(b), $(\sigma,l)$ is a \emph{$\tr$-graded Nakayama automorphism} of $\Pi$, so \cite[Thm 6.14]{grant2022serre} yields (b). For definitions of the italicised terms see \cite{grant2022serre}.
\end{proof}

To apply this to $\Pi_\dd^s$, we need a technical lemma. In view of Corollary \ref{Corollary: Truncations are derived equivalent}, we can restrict our attention to a particular cut.

\begin{defn}
Let $s=3t+1$. Define a $\zz_3$-invariant cut $K'\subset\aa^s_1$ using the following diagram. \[\includegraphics[scale=0.75]{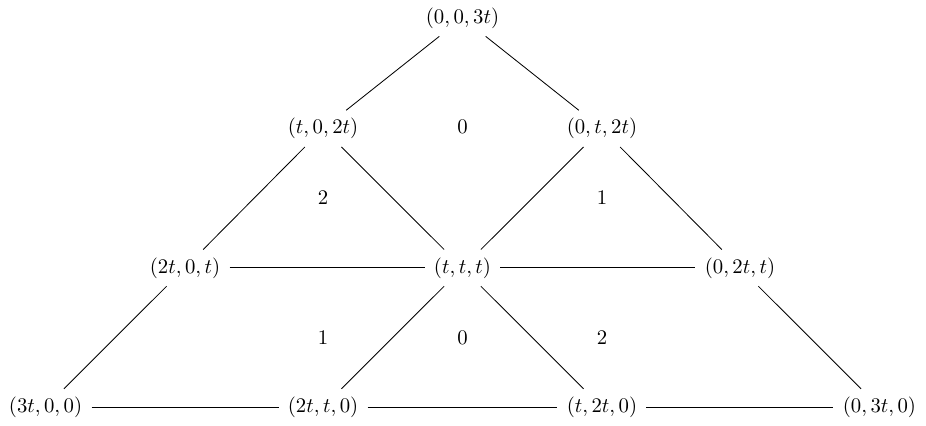}\] Here, the label $i$ indicates that all arrows $\alpha_i$ in that region (including those on the edge) should be cut. We denote by $K$ the induced cut of $(\dd^s,W_\dd^s)$, and by $g_K$ the induced grading on $\Pi_\dd^s$.
\end{defn}

See Figure \ref{Figure: Cuts} for $K$ and $K'$ in the case $t=2$. 

\begin{figure}
\centering
\includegraphics{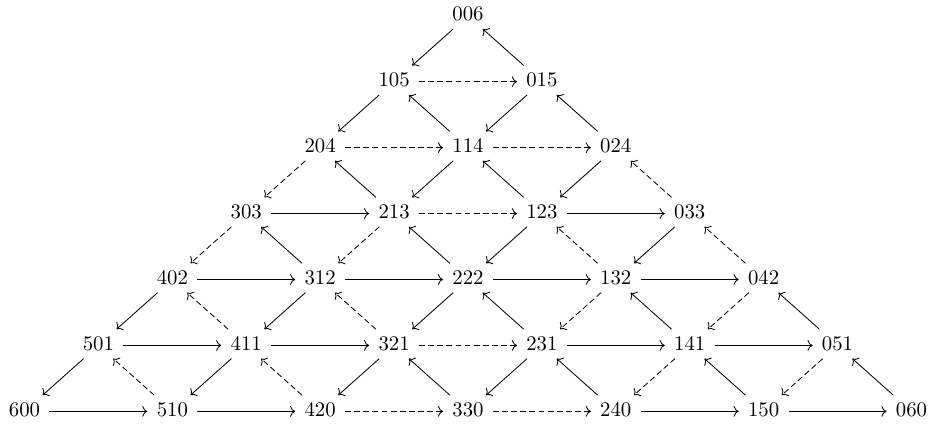}\\
\vspace{1em}
\includegraphics{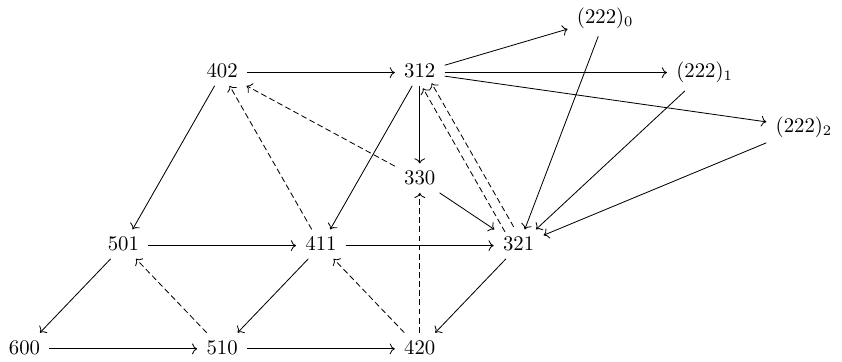}
\caption{The $\zz_3$-invariant cut $K$ of $\aa^7$ and the induced cut $K'$ of $\dd^7$, indicated by the dashed arrows.}
\label{Figure: Cuts}
\end{figure}

\begin{remark}
\label{Remark: Correspondence between cuts}
Note that all arrows in $K'$ appear in the leftmost set of \eqref{Equation: Induced cut} - see the proof of Lemma \ref{Lemma: D^s has enough cuts}. Hence, we have that \[e_x\alpha_{i,j}\in K\iff e_{\omega^j(x)}\alpha_i\in K'.\]
\end{remark}

In the following, $\soc(M)$ denotes the \textit{socle} (maximal semisimple submodule) of a module $M$. 

\begin{lemma}
\label{Lemma: Degree of longest path}
Let $s=3t+1$, $\Pi=\Pi^s_\dd$  and $x=(3t,0,0)\in\dd^s_0$. Then $\soc(e_x\Pi)=\langle p\rangle$, where \[p=
\begin{cases}
    e_x\alpha_0^\frac{3t-1}{2}\alpha_{2,1}\alpha_2^\frac{3t-1}{2} & \text{if $t$ is odd},\\
    e_x\alpha_0^\frac{3t}{2}\alpha_{2,1}\alpha_2^\frac{3t-2}{2} & \text{if $t$ is even}.
\end{cases}\]
In particular, $g_K(p)=t$.
\end{lemma}

\begin{proof}
Suppose $t\ge3$ is odd (the even case is analogous, and $t=1$ can be seen from Example \ref{Example: D^4}). The proof proceeds by recursively computing $e_x\Pi[n]$, the space of paths starting at $x$ with length $n$. Clearly $e_x\Pi[0]=\langle e_x\rangle$ and $e_x\Pi[1]=\langle e_x\alpha_0\rangle$. There are two ways to extend $e_x\alpha_0$ to a path of length 2: \[\includegraphics{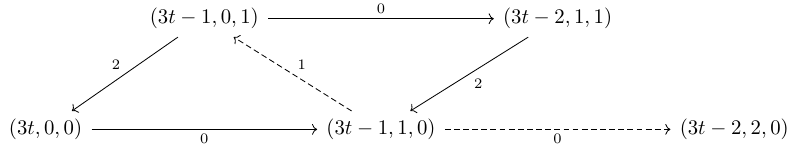}\] We have $\partial_{\alpha_2e_x}W_\dd^s=e_x\alpha_0\alpha_1$, so $e_x\alpha_0\alpha_1=0$. There is no relation reducing $e_x\alpha_0^2$, since the two arrows lie in 3-cycles which do not share a common edge. So certainly $e_x\alpha_0^2\neq0$ and $e_x\Pi[2]=\langle e_x\alpha_0^2\rangle$. 

Now let $2\le N\le\frac{3t-3}{2}$ and assume $e_x\Pi[N]=\langle e_x\alpha_0^N\rangle$. There are two ways to extend $e_x\alpha_0^N$ to a path of length $N+1$: \[\includegraphics{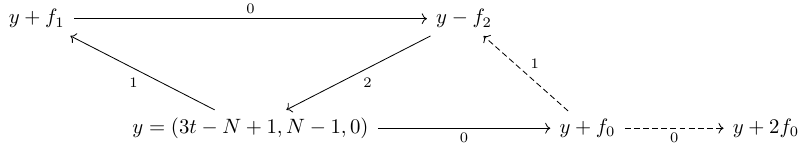}\] We have $\partial_{\alpha_2e_y}W_\dd^s=e_y(\alpha_0\alpha_1-\alpha_1\alpha_0)$. Now $e_x\Pi[N]=\langle e_x\alpha_0^N\rangle$, so $e_x\alpha_0^{N-1}\alpha_1=0$. Hence $e_x\alpha_0^N\alpha_1=e_x\alpha_0^{N-1}\alpha_1\alpha_0=0$. As before there is no relation reducing $e_x\alpha_0^{N+1}$, so this path generates $e_x\Pi[N+1]$. Therefore $e_x\Pi[n]=\langle e_x\alpha_0^n\rangle$ for all $n\le\frac{3t-1}{2}$.

There are two ways to extend $e_x\alpha_0^\frac{3t-1}{2}$ to a path of length $\frac{3t+1}{2}$: \[\includegraphics{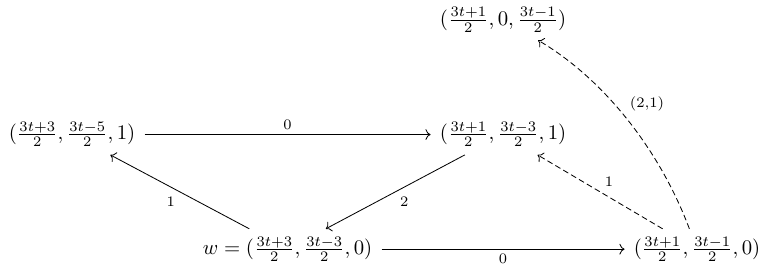}\] We have $\partial_{\alpha_2e_w}W_\dd^s=e_w(\alpha_0\alpha_1-\alpha_1\alpha_0)$, so $e_x\alpha_0^\frac{3t-1}{2}\alpha_1=e_x\alpha_0^\frac{3t-3}{2}\alpha_1\alpha_0=0$. There is no relation reducing $e_x\alpha_0^\frac{3t-1}{2}\alpha_{2,1}$ so this path generates $e_x\Pi[\frac{3t+1}{2}]$. 

There are two ways to extend $e_x\alpha_0^\frac{3t-1}{2}\alpha_{2,1}$ to a path of length $\frac{3t+3}{2}$: \[\includegraphics{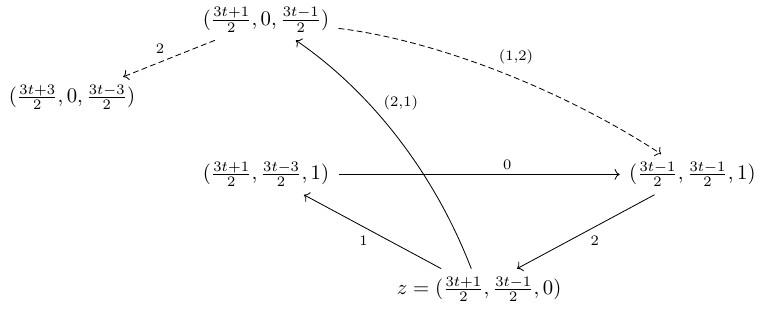}\] We have $\partial_{\alpha_2e_z}W_\dd^s=e_z(\alpha_{2,1}\alpha_{1,2}-\alpha_1\alpha_0)$, so $e_x\alpha_0^\frac{3t-1}{2}\alpha_{2,1}\alpha_{1,2}=e_x\alpha_0^\frac{3t-1}{2}\alpha_1\alpha_0=0$. There is no relation reducing $e_x\alpha_0^\frac{3t-1}{2}\alpha_{2,1}\alpha_2$ so this path generates $e_x\Pi[\frac{3t+3}{2}]$.

There are two ways to extend $e_x\alpha_0^\frac{3t-1}{2}\alpha_{2,1}\alpha_2$ to a path of length $\frac{3t+5}{2}$: \[\hspace{-0.25cm}\includegraphics{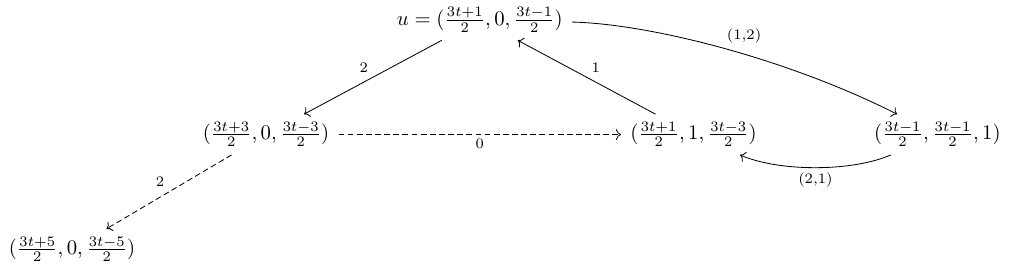}\]We have $\partial_{\alpha_1e_u}W_\dd^s=e_u(\alpha_2\alpha_0-\alpha_{1,2}\alpha_{2,1})$, so $e_x\alpha_0^\frac{3t-1}{2}\alpha_{2,1}\alpha_2\alpha_0=e_x\alpha_0^\frac{3t-1}{2}\alpha_{2,1}\alpha_{1,2}\alpha_{2,1}=0$. There is no relation reducing $e_x\alpha_0^\frac{3t-1}{2}\alpha_{2,1}\alpha_2^2$ so this path generates $e_x\Pi[\frac{3t+5}{2}]$.

Now let $2\le N\le\frac{3t-3}{2}$ and assume $e_x\Pi[\frac{3t+1}{2}+N]=\langle e_x\alpha_0^\frac{3t-1}{2}\alpha_{2,1}\alpha_2^N\rangle$. There are two ways to extend $e_x\alpha_0^\frac{3t-1}{2}\alpha_{2,1}\alpha_2^N$ to a path of length $\frac{3t+3}{2}+N$: \[\includegraphics{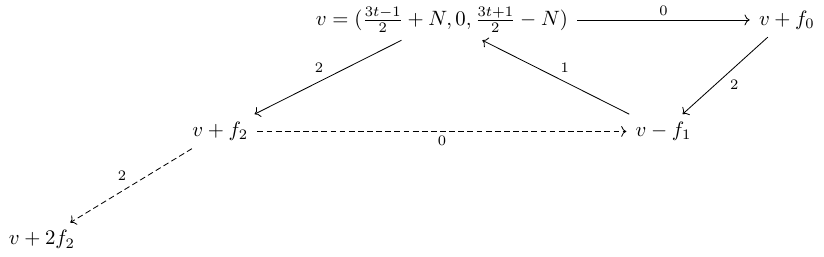}\] We have $\partial_{\alpha_1e_v}W_\dd^s=e_v(\alpha_2\alpha_0-\alpha_0\alpha_2)$, so $e_x\alpha_0^\frac{3t-1}{2}\alpha_{2,1}\alpha_2^N\alpha_0=e_x\alpha_0^\frac{3t-1}{2}\alpha_{2,1}\alpha_2^{N-1}\alpha_0\alpha_2=0$ by assumption. There is no relation reducing $e_x\alpha_0^\frac{3t-1}{2}\alpha_{2,1}\alpha_2^{N+1}$ so this path generates $e_x\Pi[\frac{3t+3}{2}+N]$. Therefore $e_x\Pi[\frac{3t+1}{2}+n]=\langle e_x\alpha_0^\frac{3t-1}{2}\alpha_{2,1}\alpha_2^n \rangle$ for all $n\le\frac{3t-1}{2}$.

Notice that $t(e_x\alpha_0^\frac{3t-1}{2}\alpha_{2,1}\alpha_2^\frac{3t-1}{2})=x$. There is only one arrow with source $x$, namely $e_x\alpha_0$ (see the first figure in the proof). But since $\partial_{\alpha_1e_{(3t-1,0,1)}}W_\dd^s=e_{(3t-1,0,1)}(\alpha_2\alpha_0-\alpha_0\alpha_2)$, we have $e_x\alpha_0^\frac{3t-1}{2}\alpha_{2,1}\alpha_2^\frac{3t-1}{2}\alpha_0=e_x\alpha_0^\frac{3t-1}{2}\alpha_{2,1}\alpha_2^\frac{3t-3}{2}\alpha_0\alpha_2=0$, so $e_x\Pi[n]=0$ for all $n>3t$.

Using Remark \ref{Remark: Correspondence between cuts}, \[g_K(p)=g_{K'}\left(e_x\alpha_0^\frac{3t-1}{2}\right)+g_{K'}\left(\alpha_2^\frac{3t+1}{2}e_x\right).\] By definition of $K'$, \[g_{K'}\left(e_x\alpha_0^\frac{3t-1}{2}\right)=\frac{t-1}{2},\ g_{K'}\left(\alpha_2^\frac{3t+1}{2}e_x\right)=\frac{t+1}{2},\] so $g_K(p)=t$.
\end{proof}

We now prove Theorem \ref{Theorem 3}.

\begin{theorem}
\label{Theorem 3 body}
Let $s=3t+1$. For any cut $C\subset\dd^s_1$, $\Pi_\dd^s/\langle C\rangle$ is fractional Calabi-Yau of dimension $2t/(t+1)$.
\end{theorem}

\begin{proof}
We prove the statement for $C=K$. This is enough, since all truncations of $\Pi_\dd^s$ are derived equivalent by Corollary \ref{Corollary: Truncations are derived equivalent}. Theorem \ref{Theorem: FInite order Nakayama implies fCY}(a) with $\sigma=\id$ and $k=1$ implies the function $l\colon\dd^s_0\to\zz$ from Lemma \ref{Lemma: l function} takes a constant value $N$. Let $x=(3t,0,0)\in\dd^s_0$. By Lemma \ref{Lemma: l function}(b), $e_x\Pi\cong D(\Pi e_x)\{-N\}$ as graded $\Pi$-modules. In Lemma \ref{Lemma: Degree of longest path} we demonstrated that $\soc(e_x\Pi)$ is generated by a path $p$ of degree $t$, so $e_x\Pi$ exists in degrees 0 through $t$. Since we also have $p\in\Pi e_x$, we deduce that $D(\Pi e_x)$ exists in degrees $-t$ through 0, and that $N=t$. Therefore $\Pi_\dd^s/\langle C\rangle$ is fractional Calabi-Yau of dimension $2t/(t+1)$ by Theorem \ref{Theorem: FInite order Nakayama implies fCY}(b).
\end{proof}

\begin{defn}
\cite[Def 1.2]{herschend2011n}
If the function $l\colon Q_0\to\zz$ from Lemma \ref{Lemma: l function} takes a constant value $N\in\zz$, we say $\Lambda$ is \emph{$(N+1)$-homogeneous}.
\end{defn}

The following is immediate from the proof of Theorem \ref{Theorem 3 body}.

\begin{corollary}
\label{Corollary: Pi_D is homogeneous}
Let $s=3t+1$. For any cut $C\subset\dd^s_1$, $\Pi_\dd^s/\langle C\rangle$ is $(t+1)$-homogeneous.
\end{corollary}

\section{2-Auslander-Reiten quivers}

\label{Section: 2AR Quivers}

The general theory in this section is from \cite[\S1]{iyama2011cluster}.

Let $\Lambda$ be a $d$-representation-finite algebra with a $d$-cluster tilting object $M\in\mod\Lambda$. Then $\add M$ is equal to the subcategory $\mm\coloneqq\add\{\tau_d^{-i}\Lambda\mid i\ge0\}\subset\mod\Lambda$ \cite[Thm 1.6]{iyama2011cluster}. We call $\mm$ the \emph{$d$-cluster tilting subcategory} of $\mod\Lambda$.

The functors $\tau_d$, $\tau_d^-$ play the role of Auslander-Reiten translates on this subcategory: they induce mutually quasi-inverse equivalences $\tau_d\colon\underline{\mm}\to\overline{\mm}$ and $\tau^-_d\colon\overline{\mm}\to\underline{\mm}$ between the stable and costable categories. 

Thus, it makes sense to consider a \emph{$d$-Auslander-Reiten quiver} of $\mod\Lambda$, whose vertices are indecomposable objects of $\mm$, and the number of arrows from $X$ to $Y$ is equal to the dimension of $\rad_\mm(X,Y)/\rad^2_\mm(X,Y)$, where $\rad$ denotes the Jacobson radical \cite[Def 6.1]{iyama2011cluster}.

\begin{example}
Let $\Pi=\Pi^4_\dd$, $C=\{\alpha_1,\alpha_{2,1}\}$ and $\Lambda=\Pi/\langle C\rangle$. To simplify notation relabel as follows. \[\includegraphics{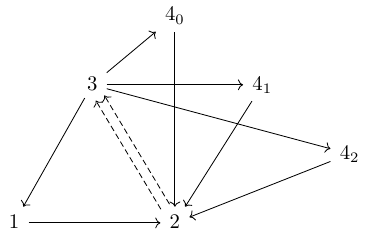}\] Using the definition $\tau_2^-=\tau^-\Omega^-$, one computes the radical series
\begin{gather*}
	e_1\Lambda=\begin{smallmatrix}1\\2\end{smallmatrix},\ e_2\Lambda=\begin{smallmatrix}2\end{smallmatrix},\ e_3\Lambda=\begin{smallmatrix}3\\1\,4_0\,4_1\,4_2\\2\,2\end{smallmatrix},\ e_{4_k}\Lambda=\begin{smallmatrix}4_k\\2\end{smallmatrix},\\
	\tau_2^-(e_1\Lambda)=\begin{smallmatrix}3\\1\end{smallmatrix},\ \tau_2^-(e_2\Lambda)=\begin{smallmatrix}3\,3\\1\,4_0\,4_1\,4_2\\2\end{smallmatrix},\ \tau_2^-(e_3\Lambda)=\begin{smallmatrix}3\end{smallmatrix},\ \tau_2^-(e_{4_k})=\begin{smallmatrix}3\\4_k\end{smallmatrix}.
\end{gather*}
The 2-Auslander-Reiten quiver of $\mm\subset\mod\Lambda$ is as follows. The dashed lines indicate the action of $\tau_2^-$. \[\includegraphics{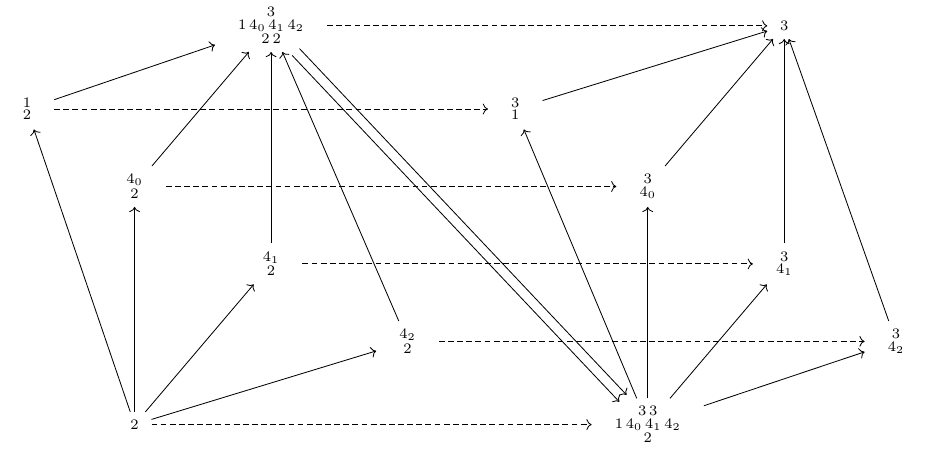}\]
\end{example}

Given a $d$-representation-finite algebra $\Lambda$, there is also a notion of higher Auslander-Reiten theory on its bounded derived category $\dd=\derb(\mod\Lambda)$. Here the role of Auslander-Reiten translate is played by the autoequivalence $\nu_d\coloneqq\nu\circ\Sigma^{-d}\colon\dd\to\dd$, and the $d$-cluster tilting subcategory is $\uu\coloneqq\add\{\nu_d^{-i}\mid i\in\zz\}\subset\dd$. This definition is justified by the fact that for any $X\in\mm$ with no injective summands, $\tau_d ^-X\cong\nu_d^- X$.

\begin{example}
The 2-Auslander-Reiten quiver of $\uu\subset\dd$ is as follows. \[\hspace{-1cm}\includegraphics[scale=0.75]{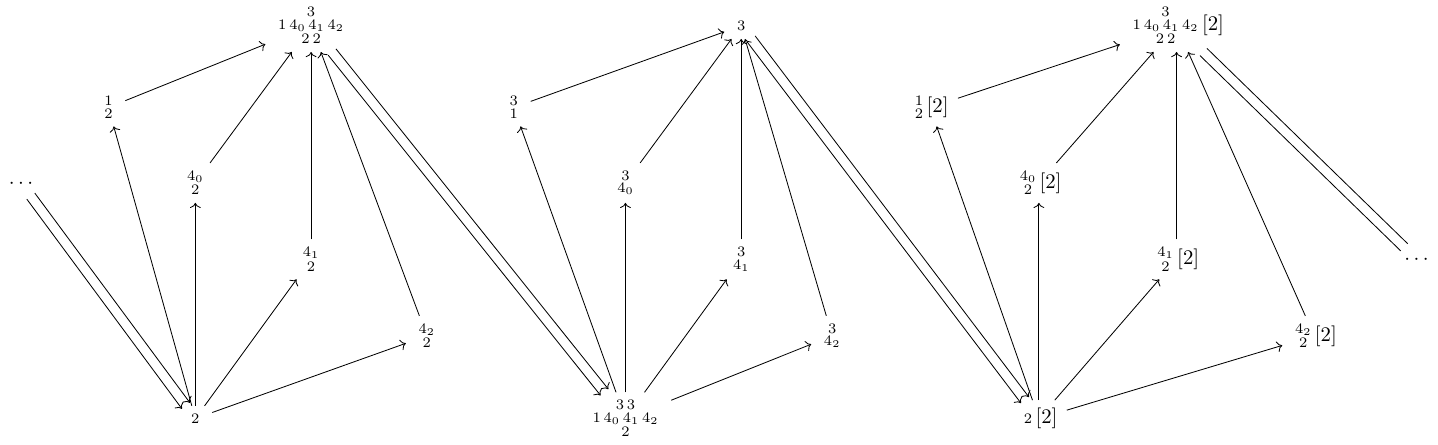}\]

On $\uu$, the Serre functor $\nu$ has orbits
\begin{gather*}
    \begin{smallmatrix}2\end{smallmatrix}\mapsto\begin{smallmatrix}3\,3\\1\,4_0\,4_1\,4_2\\2\end{smallmatrix}\mapsto\begin{smallmatrix}2
    \end{smallmatrix}[2],\\
    \begin{smallmatrix}1\\2
    \end{smallmatrix}\mapsto\begin{smallmatrix}3\\1
    \end{smallmatrix}\mapsto\begin{smallmatrix}1\\2
    \end{smallmatrix}[2],\\
    \begin{smallmatrix}4_k\\2
    \end{smallmatrix}\mapsto\begin{smallmatrix}3\\4_k
    \end{smallmatrix}\mapsto\begin{smallmatrix}4_k\\2
    \end{smallmatrix}[2],\\
    \begin{smallmatrix}3\\1\,4_0\,4_1\,4_2\\2\,2
    \end{smallmatrix}\mapsto\begin{smallmatrix}3
    \end{smallmatrix}\mapsto\begin{smallmatrix}3\\1\,4_0\,4_1\,4_2\\2\,2
    \end{smallmatrix}[2],
\end{gather*}
aligning with the fact $\Lambda$ is $2/2$-fractional Calabi-Yau.

By homogeneity (Corollary \ref{Corollary: Pi_D is homogeneous}), if $\Lambda=\Pi_\dd^{3t+1}/\langle C\rangle$, the Serre functor on $\uu$ has disjoint orbits
\[e_x\Lambda\mapsto\tau_2^-(e_x\Lambda)\mapsto\cdots\mapsto\tau^{1-t}_2(e_x\Lambda)\mapsto D(\Lambda e_x)\mapsto e_x\Lambda[2t]\]
for each $x\in\dd^{3t+1}_0$. The orbits are disjoint since $\tau_2^{-i}(e_x\Lambda)\cong\tau_2^{-j}(e_y\Lambda)$ implies $\tau_2^{j-i}(e_x\Lambda)\cong e_y\Lambda$. Since $e_y\Lambda$ is projective, we must have $i=j$, whence $x=y$ since  $\Lambda$ is basic.
\end{example}

We give a recipe to construct the 2-Auslander-Reiten quiver of a 2-representation-finite algebra of type D.

\begin{prop}
\label{Prop: Recipe for 2AR quiver}
Let $s=3t+1$, $C$ be a cut of $\Pi=\Pi^s_\dd$, and $\Lambda=\Pi/\langle C\rangle$. The 2-Auslander-Reiten quiver of $\uu\subset\derb(\mod\Lambda)$ is isomorphic to the quiver $\Gamma$ with vertices \[\Gamma_0=\{(x,i)\mid x\in\dd^s_0,\ i\in\zz\},\] where there is an arrow $(x,i)\longrightarrow(y,j)$ in $\Gamma_1$ if and only if
\begin{enumerate}
    \item $i=j$ and there is an arrow $y\longrightarrow x$ in $\dd^s_1\backslash C$, or
    \item $j=i+1$ and there is an arrow $y\longrightarrow x$ in $C$.
\end{enumerate}
The 2-Auslander-Reiten quiver of $\mm\subset\mod\Lambda$ is isomorphic to the full subquiver $\Gamma'$ of $\Gamma$ on the vertices \[\Gamma_0'=\{(x,i)\mid x\in\dd^s_0,\ 0\le i\le t\}.\]
\end{prop}

\begin{proof}
The vertices of the 2AR quiver of $\uu$ are the isoclasses of its indecomposable objects. Now
\begin{align*}
    \uu&=\add\{\nu_2^{-i}\Lambda\mid i\in\zz\}\\
    &=\add\left\{\nu_2^{-i}\left(\bigoplus_{x\in\dd_0^s}e_x\Lambda\right)\mid i\in\zz\right\}\\
    &=\add\left\{\bigoplus_{x\in\dd_0^s}\nu_2^{-i}(e_x\Lambda)\mid i\in\zz\right\}\\
    &=\add\{\nu_2^{-i}(e_x\Lambda)\mid x\in\dd^s_0,\ i\in\zz\}.  
\end{align*}
Since $\Lambda$ is basic, $\{\nu_2^{-i}(e_x\Lambda)\mid x\in\dd^s_0,\ i\in\zz\}$ is the set of isoclasses of indecomposable objects, which is clearly in bijection with the set $\Gamma_0$ defined above.

Note \cite[\S2]{herschend2011n} that $\Pi\cong\bigoplus_{r\in\zz}\hom_\uu(\Lambda,\nu_2^{-r}\Lambda)$, so for all $x,y\in\dd^s_0$ and $i,j\in\zz$, \[\hom_\uu(\nu_2^{-i}(e_x\Lambda),\nu_2^{-j}(e_y\Lambda))\cong\hom_\uu(e_x\Lambda,\nu_2^{i-j}(e_y\Lambda))\cong e_y\Pi_{j-i}e_x\] as vector spaces, where $\Pi_{j-i}$ is the degree $j-i$ part of $\Pi$.

We now show statements 1 and 2 of the proposition.
    \begin{enumerate}
        \item Suppose $i=j$. Then $\hom_\uu(\nu_2^{-i}(e_x\Lambda),\nu_2^{-i}(e_y\Lambda))$ has a basis consisting of paths from $y$ to $x$ in $\Lambda$ (since the degree zero part of $\Pi$ is $\Lambda$). In particular, the irreducible morphisms from $\nu_2^{-i}(e_x\Lambda)$ to $\nu_2^{-i}(e_y\Lambda)$ in $\uu$ (i.e. arrows $(x,i)\longrightarrow(y,i)$ in $\Gamma_1$) are in bijection with arrows $y\longrightarrow x$ in $\dd^s_1\backslash C$.
        \item Suppose $j=i+1$. Then $\hom_\uu(\nu_2^{-i}(e_x\Lambda),\nu_2^{-(i+1)}(e_y\Lambda))$ has a basis consisting of degree one paths from $y$ to $x$ in $\Pi$. In particular, the irreducible morphisms from $\nu_2^{-i}(e_x\Lambda)$ to $\nu_2^{-(i+1)}(e_y\Lambda)$ in $\uu$ (i.e. arrows $(x,i)\longrightarrow(y,i+1)$ in $\Gamma_1$) are in bijection with arrows $y\longrightarrow x$ in $C$.
    \end{enumerate}
These are all the arrows in $\Gamma_1$. Indeed, if $j<i$ then $\hom_\uu(\nu_2^{-i}(e_x\Lambda),\nu_2^{-j}(e_y\Lambda))=0$ since $\Pi$ is non-negatively graded. Conversely if $j>i+1$, then no morphism in $\hom_\uu(\nu_2^{-i}e_x\Lambda,\nu_2^{-j}(e_y\Lambda))$ is irreducible, since it corresponds to a path in $\Pi$ from $y$ to $x$ of degree $j-i>1$, which can be factorised into two paths of lower degree since $\Pi$ is generated in degree 1.

To see $\Gamma'$ is as claimed, note that its set of vertices is $\{\nu_2^{-i}(e_x\Lambda)\mid x\in\dd^s_0,\ 0\le i\le t\}$, which is clearly in bijection with the set $\Gamma_0'$ above. This is because $\Lambda$ is $(t+1)$-homogeneous, so $\tau_2^{-i}(e_x\Lambda)\cong\nu_2^{-i}(e_x\Lambda)$ for $0\le i\le t$, and $\tau_2^{-(t+1)}(e_x\Lambda)=0$.
\end{proof}

\section{Further questions}

We present some possible directions for further research in this area.
\begin{enumerate}
    \item \textbf{Basis and dimension:} Finding a basis for $\Pi_\dd^s$ can be difficult, even for small $s$.
    \begin{example}
    \label{Example: D^4}
        Let $\Pi=\Pi^4_\dd$ (see Figure \ref{Figure: Examples of D^s} and Example \ref{Example: Potentials on D^4, D^5}). Then
        \begin{align*}
            e_{300}\Pi&=\langle e_{300}, \alpha_0, \alpha_0\alpha_{2,1}, \alpha_0\alpha_{2,1}\alpha_2\rangle,\\
            e_{210}\Pi&=\langle e_{210}, \alpha_1, \alpha_{2,1}, \alpha_1\beta_0, \alpha_1\beta_1, \alpha_1\beta_2, \alpha_{2,1}\alpha_2, \alpha_{2,1}\alpha_2\alpha_0 \rangle,\\
            e_{201}\Pi&=\langle e_{201}, \alpha_2, \beta_0, \beta_1, \beta_2, \alpha_2\alpha_0, \beta_0\gamma_0, \alpha_2\alpha_0\alpha_{2,1} \rangle,\\
            e_{(111)_k}\Pi&=\langle e_{(111)_k}, \gamma_k, \gamma_k\alpha_1, \gamma_k\alpha_1\beta_k \rangle \text{ for all $k\in\{0,1,2\}$,}
        \end{align*}
        so we obtain a basis for $\Pi$. In particular, $\dim_\cc\Pi=32$.
        
        One calculates as in Lemma \ref{Lemma: Degree of longest path}, recursively computing the space of paths starting at each vertex with a given length. To make these calculations we had to derive some additional relations:
        \begin{enumerate}
            \item $\alpha_1\beta_0\gamma_0=\zeta^2\alpha_1\beta_1\gamma_1=\zeta\alpha_1\beta_2\gamma_2$ and $\alpha_{2,1}\beta_0\gamma_0=\alpha_{2,1}\beta_1\gamma_1=\alpha_{2,1}\beta_2\gamma_2$;
            \item $\beta_0\gamma_0\alpha_1=\zeta^2\beta_1\gamma_1\alpha_1=\zeta\beta_2\gamma_2\alpha_1$ and $\beta_0\gamma_0\alpha_{2,1}=\beta_1\gamma_1\alpha_{2,1}=\beta_2\gamma_2\alpha_{2,1}$;
            \item if $i\neq j$ then $\gamma_i\alpha_1\beta_j=\gamma_i\alpha_{2,1}\beta_j=0$.
        \end{enumerate}
        To see (a), note that
        \begin{gather*}
            \alpha_{2,1}\beta_2\gamma_2=-\zeta\alpha_{2,1}\beta_0\gamma_0-\zeta^2\alpha_{2,1}\beta_1\gamma_1=-\zeta\alpha_1\beta_0\gamma_0-\zeta\alpha_1\beta_1\gamma_1, \text{ but also}\\
            \alpha_{2,1}\beta_2\gamma_2=\zeta\alpha_1\beta_2\gamma_2=-\zeta^2\alpha_1\beta_0\gamma_0-\alpha_1\beta_1\gamma_1,
        \end{gather*}
        so $\zeta\alpha_1\beta_0\gamma_0+\zeta\alpha_1\beta_1\gamma_1=\zeta^2\alpha_1\beta_0\gamma_0+\alpha_1\beta_1\gamma_1$, which rearranges to give $\alpha_1\beta_0\gamma_0=\zeta^2\alpha_1\beta_1\gamma_1$. One similarly shows $\alpha_1\beta_0\gamma_0=\zeta\alpha_1\beta_2\gamma_2$, and the fact $\alpha_1\beta_k\gamma_k=\zeta^k\alpha_{2,1}\beta_k\gamma_k$ for all $k\in\{0,1,2\}$ implies the second statement.
        
        Part (b) is demonstrated analogously. To see (c), assume $i\neq j$. Then $\gamma_i\alpha_1\beta_j=\zeta^i\gamma_i\alpha_{2,1}\beta^j$, but also $\gamma_i\alpha_1\beta_j=\zeta^j\gamma_i\alpha_{2,1}\beta_j$. This is a contradiction unless $\gamma_i\alpha_1\beta_j=\gamma_i\alpha_{2,1}\beta_j=0$.
    \end{example}
    \item $\mathbf{d>2}$\textbf{:} The $(d+1)$-preprojective algebras of type A are defined for all $d\in\zz^+$ ($d=1$ recovers the classical preprojective algebra of a type A quiver). Furthermore, quotienting by the ideal generated by a cut gives a $d$-representation-finite algebra \cite[\S5]{iyama2011n}. Let $\Pi_\aa$ be a $(d+1)$-preprojective algebra of type A, for $d>2$. Some questions are:
    \begin{enumerate}
        \item Can we find a group $G$ - maybe $\zz_{d+1}$ - acting on $\Pi_\aa$ by automorphisms?
        \item Can we generalise Definitions \ref{Defn: D^s} and \ref{Defn: Potential on D^s} to give an algebra $\Pi_\dd$ that is Morita equivalent to $\Pi_\aa\#G$?
        \item Is $\Pi_\dd$ the $(d+1)$-preprojective algebra of some $d$-representation-finite algebra $\Lambda$ - maybe only when its index is 1 (mod $d+1$)?
        \item Is $\Lambda$ fractional Calabi-Yau?
    \end{enumerate}
    The methods used would likely have to be different to those used in this article. Perhaps most significantly, when $d>2$ one no longer has the Herschend-Iyama classification (Theorem \ref{Theorem: Classification of 2-rep-finite algebras}).
    
    \item \textbf{Auslander algebras:} In \cite{iyama2011cluster}, Iyama studied iterated Auslander algebras of Dynkin quivers. If one takes the Auslander algebra of a linearly oriented type A quiver, one obtains a 2-representation-finite algebra of type A. For $s=3t+1$ and a cut $C$ of $\Pi_\dd^s$, is there a type D quiver whose Auslander algebra is $\Lambda=\Pi_\dd^s/\langle C\rangle$?
    
    If $C$ contains one element of $\{e_{(t+1,t,t-1)}\alpha_1, e_{(t+1,t,t-1)}\alpha_{2,1}\}$, then it also contains the other. This is because there must be precisely one cut arrow in each of the cycles $\alpha_1\beta_0\gamma_0$ and $\alpha_{2,1}\beta_0\gamma_0$. Hence, either $\Lambda$ or $C$ contains two parallel arrows between a pair of vertices. As such, the 2-Auslander-Reiten quiver of $\mm\subset\mod\Lambda$ contains parallel arrows (see Proposition \ref{Prop: Recipe for 2AR quiver}).

    In contrast, if $Q$ is a type D quiver and $\Gamma$ is its Auslander algebra, the 2-Auslander-Reiten quiver of $\mm\subset\mod\Gamma$ contains no parallel arrows - use the recipe of \cite[Def 6.11]{iyama2011cluster}. This means $\Gamma$ and $\Lambda$ are not even Morita equivalent, since the module category of a 2-representation-finite algebra has a \emph{unique} 2-cluster-tilting subcategory \cite[Thm 1.6]{iyama2011cluster}.

    We cannot argue similarly with $\uu\subset\derb(\mod\Lambda)$. Indeed, $2$-representation-finiteness is not preserved under derived equivalence (see \cite[Rem 1.6]{herschend2011n} for a nice example). So, the question becomes: is there a type D quiver whose Auslander algebra is derived equivalent to $\Lambda$?
\end{enumerate}

\textbf{Acknowledgments:} This paper was written while the author was a PhD student at the University of East Anglia. The author thanks Joseph Grant, his supervisor, for suggesting the project and providing guidance
throughout its completion. The author also thanks Alastair King for pointers to the literature, and an anonymous referee for helpful comments and corrections.

\textbf{Version of Record:} This article has been accepted for publication, after peer review, in Algebras and Representation Theory (Springer). DOI: 10.1007/s10468-024-10297-3. 





\newpage
\bibliographystyle{alpha}
\bibliography{references.bib}

\end{document}